\documentclass[12pt,a4paper,twoside, openany]{amsart}
\usepackage{amssymb,amsmath,amsfonts,amsthm,mathrsfs}

\usepackage{bbm}

\usepackage{longtable}
\usepackage{color}
\usepackage[dvipsnames,table]{xcolor}
\usepackage{graphicx}
\usepackage{hyperref}
\usepackage{multicol}

\newcommand{\blue}{\textcolor{blue}}

\newcommand{\R}{\mathbb{R}}

\newcommand{\N}{\mathbb{N}}

\DeclareMathOperator{\Cay}{Cay}
\DeclareMathOperator{\ord}{ord}
\DeclareMathOperator{\Id}{Id}
\usepackage{fancyvrb}
\usepackage{comment}
\usepackage[toc,page]{appendix}
\usepackage{longtable}
\usepackage{geometry}
\usepackage{tikz}

\usepackage{listings}
\definecolor{codegreen}{rgb}{0,0.6,0}
\definecolor{codegray}{rgb}{0.5,0.5,0.5}
\definecolor{codepurple}{rgb}{0.58,0,0.82}
\definecolor{backcolour}{rgb}{0.95,0.95,0.92}

\usepackage{cite}

\renewenvironment{thebibliography}[1]{
	\begin{oldthebibliography}{#1}
		\setlength{\parskip}{0ex}
		\setlength{\itemsep}{0ex}
	}
	{
	\end{oldthebibliography}
}

\lstdefinestyle{mystyle}{
	backgroundcolor=\color{backcolour},   
	commentstyle=\color{codegreen},
	keywordstyle=\color{magenta},
	numberstyle=\tiny\color{codegray},
	stringstyle=\color{codepurple},
	basicstyle=\ttfamily\footnotesize,
	breakatwhitespace=false,         
	breaklines=true,                 
	captionpos=b,                    
	keepspaces=true,                 
	numbers=left,                    
	numbersep=5pt,                  
	showspaces=false,                
	showstringspaces=false,
	showtabs=false,                  
	tabsize=2
}

\theoremstyle{plain}
\newtheorem{thm}{Theorem}[section]
\newtheorem{prop}[thm]{Proposition}

\newtheorem{lemma}[thm]{Lemma}

\theoremstyle{definition}
\newtheorem{defin}[thm]{Definition}
\newtheorem{ex}[thm]{Example}
\newtheorem{rmk}[thm]{Remark}

\usepackage{textcomp}
\usepackage{float}

\usepackage{graphicx}

\graphicspath{ {./images/} }

\calclayout

\newcommand{\diag}{\operatorname{diag}}

\title{Bakry-\'Emery and Ollivier Ricci Curvature of Cayley Graphs}

\author[Cushing]{David Cushing}
\address{Department of Mathematics, The University of Manchester, Manchester, UK}
\email{david.cushing@manchester.ac.uk}
\author[Kamtue]{Supanat Kamtue}
\address{Yau Mathematical Sciences Center, Tsinghua University, Beijing, China}
\email{skamtue@tsinghua.edu.cn}
\author[Kangaslampi]{Riikka Kangaslampi}
\address{Unit of Computing Sciences, Tampere University, Finland}
\email{riikka.kangaslampi@tuni.fi}
\author[Liu]{Shiping Liu}
\address{School of Mathematical Sciences, University of Science and Technology of China, Hefei 230026, China.}
\email{spliu@ustc.edu.cn}
\author[M\"unch]{Florentin M\"unch}
\address{MPI MiS Leipzig, 04103 Leipzig, Germany}
\email{florentin.muench@mis.mpg.de}
\author[Peyerimhoff]{Norbert Peyerimhoff}
\address{Department of Mathematical Sciences, Durham University, Durham, UK}
\email{norbert.peyerimhoff@durham.ac.uk}

\date{\today}

\begin{document}

\maketitle

\begin{abstract}
    In this article we study two discrete curvature notions, Bakry-\'Emery curvature and Ollivier Ricci curvature, on Cayley graphs. We introduce Right Angled Artin-Coxeter Hybrids (RAACHs) generalizing Right Angled Artin and Coxeter groups (RAAGs and RACGs) and derive the curvatures of Cayley graphs of certain RAACHs. Moreover, we show for general finitely presented groups $\Gamma = \langle S \, \mid\, R \rangle$ that addition of relators does not lead to a decrease the weighted curvatures of their Cayley graphs with adapted weighting schemes.  
\end{abstract}

\tableofcontents

	\section{Introduction}

In recent decades, various notions of Ricci curvature were introduced and studied on discrete spaces like combinatorial and weighted graphs. Two natural curvature concepts are due to Ollivier and Bakry-\'Emery and their computation can be described as a specific linear optimization problem. 
They are the curvature notions considered in this paper. An interesting class of graphs are Cayley graphs, which are metric realisations of finitely generated groups. 
Various papers are concerned with curvature properties of certain Cayley graphs (see, e.g., \cite{ChY96,LY10,LLY11,BHLLMY15,KKRT16,EHMT17,Mu18,BS19,CLP20,KLM20,Sic20,Sic21,RS23}). For example, it is well known that abelian Cayley graphs have both non-negative Ollivier Ricci and Bakry-\'Emery curvature (see \cite{LY10,KKRT16} and also \cite{CKKLP21} with its references). 
Since Cayley graphs are vertex transitive and both curvature notions are local and invariant under graph isomorphisms, it suffices to study curvature properties of Cayley graphs near the identity. In this paper, we provide curvature results for the Cayley graphs of some Right Angled Artin-Coxeter Hybrids (RAACHs), which are a generalization of Right Angled Coxeter Groups (RACGs)  and Right Angled Artin Groups (RAAGs). We study also the effect on the Cayley graph curvatures when adding relators in finitely presented groups. Generally, the Cayley graphs become more connected under this process and one would expect that their curvatures do not decrease. In this introduction, we discuss briefly the relevant concepts, present our results and provide an outline of the paper.

\subsection{Concepts}

Cayley graphs are vertex transitive graphs associated to finitely generated groups. Let $\Gamma$ be a group with a finite set $S$ of non-trivial generators. Henceforth, we use the notation $S^* = \{s,s^{-1} \mid s \in S \}$ for the symmetrized set of generators. The associated \emph{Cayley graph} is a combinatorial graph, denoted by $\Cay(\Gamma,S)$, with vertex set $\Gamma$ and an edge between $\gamma, \gamma' \in \Gamma$ if and only if we have $\gamma' = \gamma s$ for some $s \in S^*$. $\Gamma$ acts transitively on $\Cay(\Gamma,S)$ by left multiplication. For general combinatorial graphs we use the notation $G=(V,E)$, where $V$ denotes the vertex set of $G$ and $E$ denotes its set of edges. We call $x,y \in V$ \emph{neighbours}, if they are connected by an edge and we write $x \sim y$. Connected combinatorial graphs have a natural \emph{distance function} $d: V \times V \to \N \cup \{0]$, and spheres and ball around a vertex $x \in V$ are defined by $S_k(x) = \{ y \in V \mid d(x,y) = k\}$ and $B_k(x) = \{y \in V \mid d(x,y) \le k \}$, respectively.
All graphs in this paper will be simple, that is, they do not have loops or multiple edges, but they may be weighted.

We like to emphasize that, in this paper, Cayley graphs associated to $\Gamma$ and $S$ are always simple, even though $\gamma' = \gamma s$ implies $\gamma = \gamma' s^{-1}$ with $s,s^{-1} \in S^*$, and readers might therefore be tempted to insert two edges connecting $\gamma$ and $\gamma'$ in the case $\ord(s) \ge 3$. However, such a relation between $\gamma, \gamma' \in \Gamma$ leads to only one indirected edge in $\Cay(\Gamma,S)$ in this paper.
In the context of Theorem \ref{thm:addrel} below, we will encounter the issue of merging and collapsing of generators, but we will then take care of this phenomenon by keeping the underlying Cayley graph simple and dealing with the merging of generators by introducing weights. Readers may think that collapsing of generators would lead to loops, but we simply ignore these collapsed generators and therefore avoid any loops in the corresponding Cayley graphs. These comments are meant to avoid any confusion about Cayley graphs now and later in this introduction.

A combinatorial graph $G$ is called  \emph{weighted}, if it comes with a vertex measure $m: V \to (0,\infty)$ and/or edge weights $w: E \to (0,\infty)$. Sometimes it is convenient to consider edge weights as a symmetric function $w: V \times V \to [0,\infty)$ where we have $w(x,y) = w(y,x)$, and $w(x,y) = 0$ if and only if $x \not\sim y$.
The \emph{degree} of a vertex $x \in V$, denoted by $\deg(x)$, is the number of its neighbours in the case of an unweighted graph, and defined as $\deg(x) = \sum_{y \sim x} w(x,y)$ in the case of a weighted graph $(G,m,w)$. The \emph{Laplacian} on a weighted graph $(G,m,w)$ is defined on functions $f: V \to \R$ as follows:
\begin{equation} \label{eq:wlap} 
\Delta f(x) = \frac{1}{m(x)} \sum_{y \in V} w(x,y)(f(y)-f(x)). \end{equation}
In the case of an unweighted graph we choose $m \equiv \mathbbm{1}_V$ and $w(x,y) =1$ for $x \sim y$. We refer to this Laplacian as the \emph{non-normalized Laplacian}. Another possible choice is $m(x)=\deg(x)$ and $w(x,y)=1$ for $x \sim y$, in which case the associated Laplacian is called the \emph{normalized Laplacian}. If $m(x) \ge \sum_{y \in V} w(x,y)$ for all $x \in V$, we can think of $\mu_x(y) = \frac{w(x,y)}{m(x)}$ as transition probabilities of a random walk with laziness $\mu_x(x) = 1 - \sum_{y \sim x} \mu_x(y)$, and we call the associated Laplacian the \emph{random walk Laplacian}. The normalized Laplacian coincides with the Laplacian associated to the simple random walk without laziness.
Note that $-\Delta$ has only real eigenvalues and is non-negative. We denote its eigenvalues in increasing order by
$$ 0 = \lambda_1(-\Delta) \le \lambda_2(-\Delta) \le \cdots \le\lambda_n(-\Delta), $$
where $n$ is the number of vertices of $G$. For a weighted graph with missing vertex measure or missing edge weigths, we choose the trivial vertex measure $m \equiv \mathbbm{1}_V$ or trivial edge weights $w \equiv \mathbbm{1}_E$ for the corresponding Laplacian, unless stated otherwise.

Let us now introduce a special family of finitely presented groups which we call \emph{Right Angled Artin-Coxeter Hybrids (RAACHs)}. A RAACH $\Gamma$ is determined by a finite weighted graph $(H,m)$ with vertex set $S$ and a function $m: S \to \{2,3,4,\infty\}$. The elements of $S$ are the generators of $\Gamma$ and their orders are determined by $m$. Henceforth, we use the notation
$$ R_j := \{ s \in S\, \mid\, \ord(s) = j \} $$
for the generators of order $j \in \{2,3,4,\infty\}$. By abuse of notation, we will also use the same symbol $R_j$ for the set of relations $\{ s^j = e\, \mid \, s \in R_j \}$. The only other defining relations of a RAACH $\Gamma$ are commutator relations $[s,t] = s^{-1} t^{-1} s t = e$, which are present if and only if $s,t \in S$ are neighbours in $H$. Collecting all this commutator relations into the set $C_H$, we define the RAACH $\Gamma$ as
$$ \Gamma = \langle S \, \mid \, R_2 \cup R_3 \cup R_4 \cup R_\infty \cup C_H \rangle. $$
In the case $R_\infty = S$, $\Gamma$ is a Right Angled Artin Group (RAAG), and in the case $R_2 = S$, $\Gamma$ is a Right Angled Coxeter Group (RACG). We refer to the graph $(H,m)$ as the \emph{defining graph} of the RAACH $\Gamma$. 

Let us finish with some background about the two curvature notions considered in this paper.

\emph{Ollivier Ricci curvature} is motivated by the following fact in Riemannian Geometry: positive/negative Ricci curvature implies, in the setting of Riemannian manifolds, that the average distance between corresponding points in closeby small balls is smaller/larger than the distance between their centers (see \cite{RS05}). Ollivier transferred this property to metric spaces with random walks in \cite{Oll09} with the help of probability measures on balls in the language of Optimal Transport Theory. In the setting of graphs, Ollivier Ricci curvature $\kappa_p$ depends on an idleness parameter $p \in [0,1]$ and can be understood as a curvature notion on their edges.  In this paper we will use a slight variation of this curvature notion due to Lin-Lu-Yau \cite{LLY11}, defined by 
$$ \kappa_{LYY}(x,y) = \lim_{p \to 1} \frac{\kappa_p(x,y)}{1-p} $$ 
for adjacent vertices $x \sim y$ representing an edge. We will also use an alternative limit-free description of $\kappa_{LLY}$ by M\"unch-Wojciechowski \cite{MW19}
using the random walk Laplacian $\Delta$, which is based on the duality principle:
$$ \kappa_{LLY}(x,y) = \inf_{\substack{f \in \textrm{\rm{1}--{\rm Lip}}\\ f(y)-f(x) = 1}} \Delta f(x) - \Delta f(y). $$
The reformulation of Ollivier Ricci curvature via Laplacians allows to consider Ollivier Ricci curvature for general weighted graphs and to remove the original restriction to probability measures.

\emph{Bakry-\'Emery curvature} was introduced in \cite{BE84}, and it is motivated by the curvature-dimension inequality. This inequality is closely related to Bochner's formula, a fundamental analytic identity in Riemannian Geometry involving the Laplacian. Its translation into the setting of graphs has been discussed in many papers (see, e.g., \cite{Elw91}, \cite{Sch99} and \cite{LY10}) and leads to a curvature notion on the vertices of a graph associated to a dimension parameter $\mathcal{N} \in (0,\infty]$. In this paper, we restrict our considerations to the choice $\mathcal{N} = \infty$ and denote 
the Bakry-\'Emery curvature of a vertex $x \in V$ by $K(x)$. 
It was discovered independently in \cite{Sic20,Sic21} and in \cite{CKLP22} that the Bakry-\'Emery curvature $K(x)$ agrees with the minimal eigenvalue of a certain symmetric matrix $A(x)$, that is
$$ K(x) = \lambda_{\min}(A(x)). $$
This viewpoint reduces the computation of Bakry-\'Emery curvature from a semidefinite programming problem to an eigenvalue problem, which is numerically much easier to handle. The matrix $A(x)$, which we call the \emph{curvature matrix} at the vertex $x\in V$, is derived in the graph setting from the above-mentioned curvature-dimensional inequality by a manipulation involving the Schur complement. We will use this eigenvalue description in the proof of our Bakry-\'Emery curvature results  for certain RAACHs. 

Both curvature notions are local concepts, that is, we only need to know about a small neighbourhood of a vertex or an edge in the graph to compute the curvature. In the case of Bakry-\'Emery curvature, this neighbourhood of a vertex $x$ is the incomplete $2$-ball $\mathring{B}_2(x)$, that is, the induced subgraph of the $2$-ball around $x$ with all edges between two vertices in the $2$-sphere removed (see, e.g., \cite{CLP20,CKPWM20}). Since Cayley graphs are vertex transitive, they have constant Bakry-\'Emery curvature and the Ollivier curvatures of edges around each vertex are the same. 

\subsection{Results}

Let us now provide a description of the main results in this paper. Our first result can be viewed as an amusing general fact about Right Angles Artin-Coxeter Hybrids (RAACHs).

\begin{prop}[Elimination of $R_4$ and $R_\infty$]
\label{prop:elimR4Rinf}
Let $\Gamma$ be a RAACH with generating set 
$$ S = R_2 \cup R_3 \cup R_4 \cup R_\infty. $$ 
Then there exists another RAACH $\Gamma'$ with generating set 
$$ S' = R_2' \cup R_3' $$
such that the Cayley graphs $G = \Cay(\Gamma,S)$ and $G' = \Cay(\Gamma',S')$ are isomorphic and have therefore the same Bakry-\'Emery and Ollivier Ricci curvatures.
\end{prop}

While we do not utilize this fact in our curvature investigations, it tells us that it suffices to investigate only Cayley graphs of RAACHs with all generators of order $2,3$, since all generators of other orders can be eliminated by specific processes described in the proof of the proposition. Note however, that only the Cayley graphs $G$ and $G'$ in the proposition are isomorphic and not the underlying groups $\Gamma$ and $\Gamma'$.

For our curvature results, we need to introduce one more concept, the \emph{associated pair} $(H^*,w)$ of a RAACH $\Gamma$ with generating set $S$. $H^*$ is a combinatorial graph with vertex set $S^* = \{s,s^{-1}\, \mid \, s \in S \}$. The edges of $H^*$ are determined by $w: S^* \times S^* \to \{0,1,2\}$, which is defined as follows:
\begin{equation}
    w(s,t) = \begin{cases} 1, & \text{if $s$ and $t$ commute and $s \neq t,t^{-1}$,} \\ 1, & \text{if $t = s^{-1}$ and $\ord(s)=4$,} \\
    2, & \text{if $t = s^{-1}$ and $\ord(s)=3$,} \\
    0, & \text{otherwise.}
    \end{cases}
\end{equation}
Two vertices $s,t \in S^*$ are adjacent in $H^*$ if and only if $w(s,t) > 0$. 

In other words, the vertices of $H^*$ are the extension of the generating set $S$ of $\Gamma$ to the corresponding symmetric generating set $S^*$, and edges in $H$ corresponding to $[s,t] = e$ for $s,t \in S$ give rise to at most four weight $1$ edges $\{s^{\pm 1},t^{\pm 1}\}$ in $H^*$. This is in accordance with the fact that $[s,t] =e$ implies $[s^{\pm 1},t^{\pm 1}] = e$. Furthermore, there are edges of weight $1$ or $2$ in $H^*$ between $s \in S$ and $s^{-1}$ if $\ord(s)=4$ or $\ord(s)=3$, respectively.

In the special case of a Right Angled Coxeter Group (RACG) with defining graph $(H,m)$, we have $m \equiv 2$, and the associated pair $(H^*,w)$ consists of the same combinatorial graph $H^* = H$ without new vertices and edges and with trivial edge weights, that is, $w \equiv \mathbbm{1}$. 

Throughout the paper, we denote the set of integers from $1$ to $n$ by $[n]$. 

We have the following explicit curvature results.

\begin{thm}[Ollivier Ricci curvature for RAACHs]
\label{thm:mainraachor}
  Let $\Gamma$ be a RAACH with generating set $S$, defining graph $(H,m)$ and associated pair $(H^*,w)$. Then we have for any $s \in S^*$, representing an edge in $G = \Cay(\Gamma,S)$ incident to the identity $e \in \Gamma$ and representing a vertex in $H^*$,
  $$ \kappa_{LLY}(s) = \frac{a+2\deg_{H^*}(s)}{D} - 2 $$
  with $D = \deg_G(e) = |S^*|$ and 
  $$ a = \begin{cases} 4, & \text{if $\ord(s) \neq 3$,} \\ 3, & \text{if $\ord(s) = 3$.} \end{cases} $$

  In the particular case of $R_2=S$ (that is, $\Gamma$ is a RACG and $H^*=H$), we have for any $s \in S$, representing an edge in $G= \Cay(\Gamma,S)$ incident to $e \in \Gamma$ and representing a vertex in $H$,
  $$ \kappa_{LLY}(s) = \frac{4+2\deg_H(s)}{D} - 2 $$
  with $D = \deg_G(e) = |S|$.
\end{thm}

\begin{thm}[Bakry-\'Emery curvature for certain RAACHs] 
\label{thm:mainraachbe}
Let $\Gamma$ be a RAACH with generating set $S$, defining graph $(H,m)$ and associated pair $(H^*,w)$. Let $G = \Cay(\Gamma,S)$ and $D = \deg_G(e) = |S^*|$. We enumerate the elements in the $1$-sphere $S_1(e)$ around the identity $e \in \Gamma$ as follows:
$$ S_1(e) = S^* = \{t_1,\dots,t_D\}, $$
with $\ord(t_i) = 3$ for $i \in [\ell]$ and $\ord(t_j) \neq 3$ for $j \in [D] \setminus [\ell]$. Then the curvature matrix at $e \in \Gamma$ is given by
$$ A(e) = (2-D) \Id_D + J - \Delta_{H^*} + \frac{1}{2} \diag(\underbrace{1,\dots,1}_{\ell},\underbrace{0,\dots,0}_{D-\ell}), $$
where $J$ is the $D \times D$ all-one matrix, $\Delta_{H^*}$ is Laplacian of the weighted graph $(H^*,w)$ with $m \equiv 1$, defined in \eqref{eq:wlap}, and $\diag(a_1,\dots,a_n)$ is the $n \times n$ diagonal matrix with diagonal entries $a_1,\dots,a_n$.
Moreover, we have the following:
\begin{itemize}
    \item[(a)] If $R_3 = \emptyset$ and $D \ge 2$, we have
    $$ 
    K(e) = 2 - D + \lambda_2(-\Delta_{H^*}).
    $$
    \item[(b)] If $R_3 = S$ and $D \ge 4$, we have
    $$ K(e) = \frac{5}{2} - D + \lambda_2(-\Delta_{H^*}). $$
\end{itemize}
\end{thm}

\begin{rmk} \label{rmk:nonnormnorm}
    The curvatures in Theorem \ref{thm:mainraachbe} are based on the choice of the non-normalized Laplacian. If one chooses the normalized Laplacian instead, the corresponding curvatures rescale by a factor $1/D$, since $G$ is $D$-regular.
\end{rmk}

We illustrate the above results with the following examples.

\begin{ex}[Cayley graphs of Coxeter groups] Coxeter groups $\Gamma$ are of the form
$$ \Gamma = \langle S\, \mid\, (s_is_j)^{m_{ij}} = e\,\, \text{for $i,j \in [D]$} \rangle $$
with generating set $S = \{s_1,\dots,s_D\}$, $m_{ii} = 1$ for $i \in [D]$ (that is, all generators are of order $2$), and $m_{ij} \in \{2,3,\dots,\infty\}$ for $i,j \in [D]$, $i \neq j$. Note that $m_{ij}=2$ means that the generators $s_i$ and $s_j$ commute. The corresponding Coxeter diagram has $D$ vertices representing the elements of $S$ with an edge between $s_i$ and $s_j$ if and only if $s_i$ and $s_j$ do not commute. The Cayley graph $G = \Cay(\Gamma,S)$ of a Coxeter group has only cycles of even length and, since our curvatures are local concepts, any cycles of length $\ge 6$ do not change the curvature values. Therefore, it suffices to restrict our considerations to Coxeter groups with $m_{ij}$ only taking values $2$ or $\infty$, that is, RACGs. We note that the defining graph $H=H^*$ of a Coxeter group $\Gamma$ is the complement of the corresponding Coxeter diagram, which we denote therefore by $H^c$. As stated in \cite[Theorem 9.6]{CLP20}, the Bakry-\'Emery curvature of the Cayley graph of a general Coxeter group with Coxeter diagram $H^c$ is given by
\begin{equation} \label{eq:BEcurvCox} 
K(e) = 2-\lambda_{\rm{max}}(-\Delta_{H^c}). 
\end{equation}
Since 
$\Delta_H + \Delta_{H^c} = J - D \Id_D$,
we conclude that
$$ \lambda_{\rm{max}}(-\Delta_{H^c}) = D - \lambda_2(-\Delta_H), $$
showing that \eqref{eq:BEcurvCox} agrees with the curvature formula in Theorem \ref{thm:mainraachbe}(a) in the case of RACGs. We like to mention that another independent Bakry-\'Emery curvature description of Cayley graphs of Coxeter groups was given in \cite{Sic20,Sic21}. Bakry-\'Emery curvatures for Hasse diagrams of Bruhat orders of finite Coxeter groups can be found in \cite{Sic22}, and for Bruhat graphs of finite Coxeter groups in \cite{Sic23}.
\end{ex}

\begin{ex}[Regular triangle trees] Let $H$ be the disjoint union of $D_0 \ge 2$ isolated vertices $S = \{s_1,\dots,s_{D_0}\}$ without edges and $m(s) = 3$ for all $s \in S$. The Cayley graph $G = \Cay(\Gamma,S)$ of the RAACH $\Gamma$ with defining graph $(H,m)$ has a treelike structure with triangles as building blocks (see Figure \ref{fig:triangle-graphs}). In fact, the only cycles of $G$ are triangles, every edge of $G$ is contained in a (unique) triangle, and every vertex has degree $2D_0$. 
The associated pair $(H^*,w)$ is the disjoint union of $D_0$ copies of $K_2$ with all edges of weight $2$. Since $H^*$ is not connected, we have $\lambda_2(-\Delta_{H^*}) = 0$. Consequently, all edges of $G$ have constant Ollivier Ricci curvature
$$ \kappa_{LLY} = \frac{7}{2D_0} - 2 $$
and constant Bakry-\'Emery curvature
$$ K = \frac{5}{2D_0} - 2, $$
where we used the normalized Laplacian for comparison reasons (see Remark \ref{rmk:nonnormnorm}).
 
\begin{figure}[h]
\includegraphics[width=0.49\textwidth]{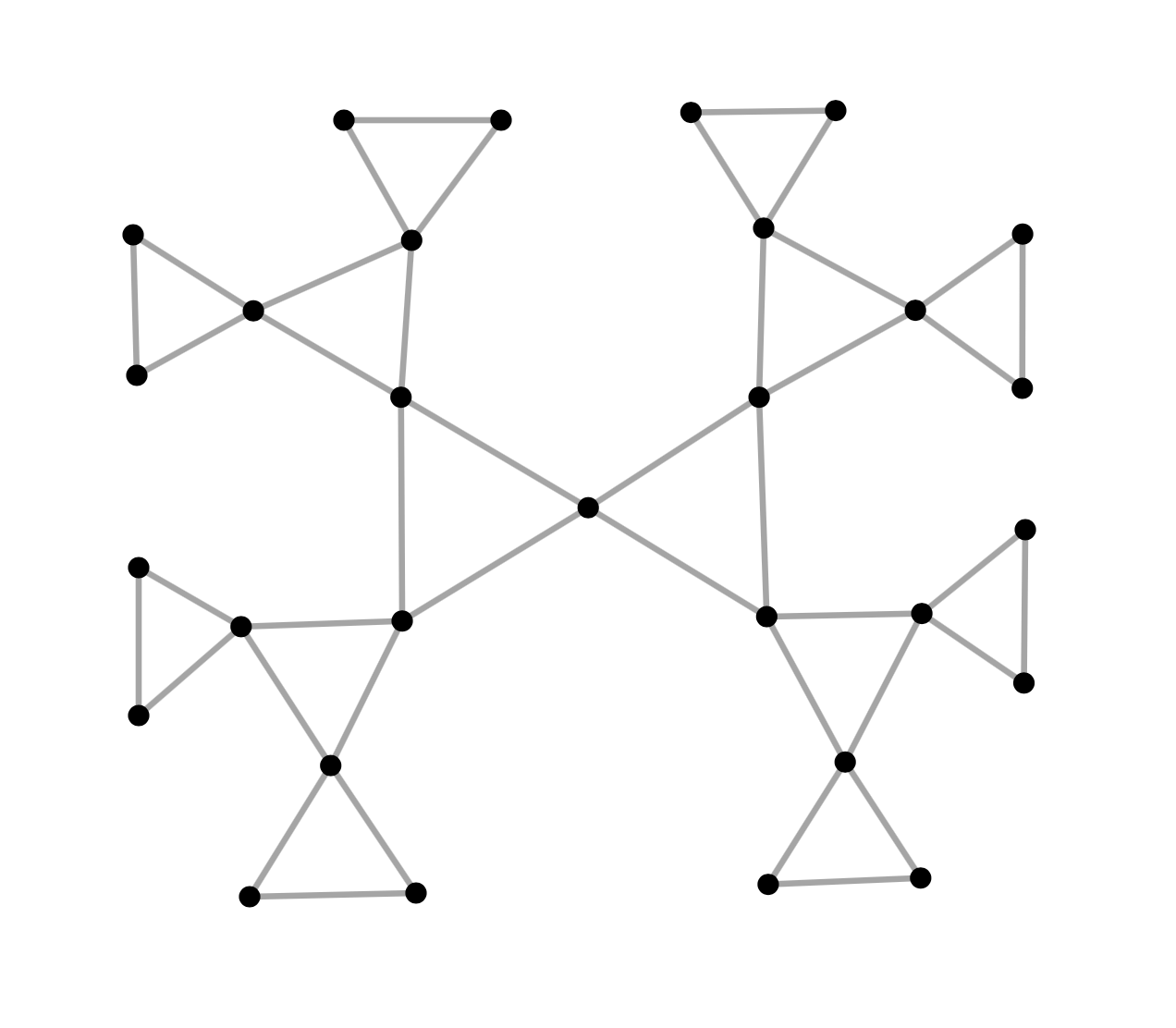}
\includegraphics[width=0.49\textwidth]{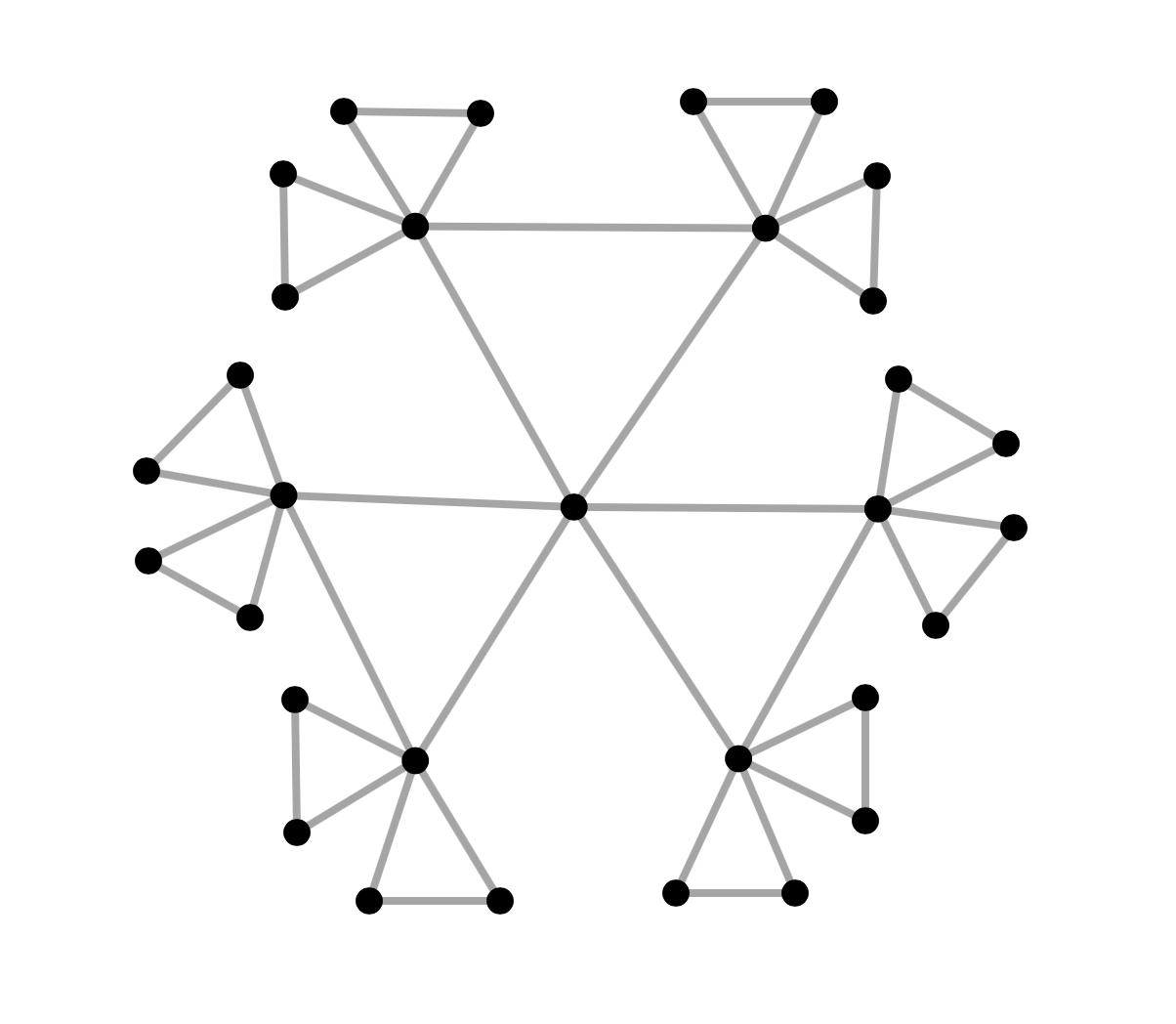}
\caption{Local structures of $2D_0$-regular triangle trees 
for $D_0=2$ and $D_0=3$.}
\label{fig:triangle-graphs}
\end{figure}
\end{ex}

\begin{ex}[Regular Trees] \label{ex:regtrees} The Cayley graphs of RAACHS $\Gamma$ with generating sets $$ S = R_2 = \{s_1,\dots,s_D\} \,\, (D \ge 2) $$
or 
$$ S = R_\infty = \{s_1,\dots,s_{D_0}\} \,\, (D_0 \ge 1) \quad \text{and $D=2D_0$} $$
and defining graphs $(H,m)$ without edges are $D$-regular trees. They have constant Ollivier Ricci curvature $\kappa_{LLY} = \frac{4}{D} -2$ and constant Bakry-\'Emery curvature $K = \frac{2}{D}-1$, based on the normalized Laplacian (see Remark \ref{rmk:nonnormnorm}). Adding edges into $H$ means also adding edges into $H^*$ and adding commutator relations into $C_H$. It is straightforward to see from the formulas in Theorem \ref{thm:mainraachor} that Ollivier Ricci curvature does not decrease under this process. It follows also from Theorem \ref{thm:mainraachbe}(a) that Bakry-\'Emery curvature does not decrease under the addition of commutator relations into $C_H$, since 
$\lambda_2(-\Delta_{H^*}) = 0$ for $H^*$ without edges, and since the addition of any edge $\{ t_i,t_j\}$ (of weight $1$) into $H^*$  with vertex set $S^* = \{t_1,\dots,t_D\}$ translates into adding a special non-negative matrix $M_{ij}$ into the matrix representation of $-\Delta_{H^*}$. The matrix $M_{ij}$ has non-zero entries only in the positions $(i,i)$, $(i,j)$, $(j,i)$ and $(j,j)$, where it is of the form 
$$ \begin{pmatrix} 1 & -1 \\ -1 & 1 \end{pmatrix}, $$
which is obviously a positive semidefinite matrix.
\end{ex}

In Example \ref{ex:regtrees}, we observed that the addition of commutator relations in certain RAACHs does not decrease the curvatures of their Cayley graphs. Our final result states such a fact regarding the addition of general relators in arbitrary Cayley graphs (see also \cite[Conjecture 9.3]{CLP20}). Since the precise formulation of this result is subtle, we need to first have a deeper look in the concept of Cayley graphs and to extend some of our notation. Let $\Gamma$ be finitely presented as $\Gamma = \langle S \mid R \rangle$. We refer to $S$ as the \emph{alphabet} of $\Gamma$. 
The set $R$ is a finite set of words in $S^* = \{s,s^{-1} \mid s \in S\}$. The elements of $R$ are called \emph{relators} in the presentation of  $\Gamma$. The group $\Gamma$ is then given by equivalence classes $[w]_R$ of words $w$ with letters in $S^*$ modulo the relators in $R$ and $\Cay(\Gamma,S)$ is the corresponding Cayley graph. The group $\Gamma$ can be understood as the quotient of the free group $F_S$ in $S$ with no relations other than $s s^{-1} = s^{-1} s =e$ by the normal closure of the set $R$. The normal closure of $R$ is the smallest normal subgroup of $F_S$ containing $R$.  
In the presentation of $\Gamma$ and its Cayley graph, we used a slight abuse of notation, since the elements of $S$ are not elements of $\Gamma$, but only their equivalence classes. We will use the notation $S_\Gamma$ for the set of \emph{non-trivial} equivalence classes $[s]_R$ corresponding to $s \in S$. The set $S_\Gamma$ may have a smaller cardinality than $S$, and its elements are \emph{generators} of the group $\Gamma$. The set $S_\Gamma^*$ is defined similarly as the set of \emph{non-trivial} equivalence classes $[s]_R$ corresponding to $s \in S^*$. Strictly speaking, we would need to write $\Gamma = \langle S_\Gamma \mid R \rangle$ and $\Cay(\Gamma,S_\Gamma \rangle$ for the group presentation and its associated Cayley graph. We will use the simpler notation if the context is clear and only use the more precise notation $\Gamma = \langle S_\Gamma | R \rangle$ and $\Cay(\Gamma,S_\Gamma)$ when we want to avoid confusion. 

Let $\Gamma' = \langle S \mid R' \rangle$ be second group with the same alphabet and a larger set $R'$ of relators, that is $R' \supset R$. Since the elements of $\Gamma$ (and of $\Gamma'$) are equivalence classes of words defined via the relators, we have a canonical map $\Phi: \Gamma \to \Gamma'$, mapping the equivalence class $[w]_R$ to the corresponding equivalence class $[w]_{R'}$. This map induces a $1$-Lipschitz map (with respect to the graph distances) on the corresponding Cayley graphs $G = \Cay(\Gamma,S)$ and $G' = \Cay(\Gamma',S)$, which we also denote by $\Phi$. Note that the larger set of relators in $R'$ may lead to various identifications of different group elements under the map $\Phi$, and different generators of $\Gamma$ may be identified in $\Gamma'$ or even collapse to the identity. 
Therefore, it is difficult to understand the nature of this map $\Phi$ geometrically, and even its local behaviour is non-trivial. Note also that, by the very definition, we collect only \emph{non-trivial} equivalence classes $[s]_{R'}$ in the sets $S_{\Gamma'}$ and $S_{\Gamma'}^*$. A natural guess is that the curvatures do not decrease under this process, that is, 
$$ K(x') \ge K(x) $$ 
for vertices $x$ and $x'$ in the Cayley graphs $G$ and $G'$ and $$ \kappa_{LLY}(x',y') \ge \kappa_{LLY}(x,y)$$ 
for corresponding edges $x \sim_G y$ and $x' \sim_{G'} y'$ under the map $\Phi$. The following example shows, however, that the curvature monotonicity assumption is not always true in this original form.

\begin{ex}
    Let $S = \{a,b\}$ and $R = \{ a^4, b^{-1}a^2 \}$. Then the group $\Gamma = \langle S\, |\, R \rangle$ is cyclic of order $4$ (with generator $a$), and the Cayley graph $\Cay(\Gamma,S)$ is isomorphic to the complete graph $K_4$ with constant Bakry-\'Emery curvature $K(x) = 3$ and constant Ollivier Ricci curvature $\kappa_{LLY}(x,y)=4$. Setting $R' = R \cup \{ a^2 \}$ and $\Gamma' = \langle S\, |\, R' \rangle$, the corresponding Cayley graph $\Cay(\Gamma',S)$ is isomorphic to $K_2$ with 
    $K(x') = 2$ and $\kappa_{LLY}(x',y') = 2$. These curvature values can be easily verified with the graph curvature calculator (see \cite{CKLLS19} and the freely accessible web-app at \url{https://www.mas.ncl.ac.uk/graph-curvature/}). In this case both curvature types decrease under the transition from $\Gamma$ to $\Gamma'$.
\end{ex}

Let us have a closer look at this example. The equivalence classes of the generators $[a]_R, [a^{-1}]_R, [b]_R = [b^{-1}]_R$ in $\Gamma$ are different and non-trivial. However, this is not true for their $\Phi$-images. The equivalence class $[b]_{R'}$ is trivial (the collapse of a generator) and the equivalence classes $[a]_{R'}$ and $[a^{-1}]_{R'}$ coincide (merging of generators). It turns out that the violation of our above curvature monotonicity assumption is not caused by the collapse of generators but by the merging of generators. Since there is a one-one correspondence between the generators and the edges incident to any Cayley graph vertex, we take care of this merging of generators by introducing a weighting scheme on our Cayley graphs, and by assigning increased weights on edges which are merged, and to use
weighted Laplacians and their corresponding weighted Bakry-\'Emery and Ollivier Ricci curvatures. In the case of our Cayley graph $G = \Cay(\Gamma,S)$, we choose the trivial vertex measure $m = \mathbbm{1}_\Gamma$ and edge weight functions $w$ which are invariant under the $\Gamma$-left action. This means that $w$ is determined by a function $w_0: S^*_\Gamma \to (0,\infty)$ satisfying $w_0(s) = w_0(s^{-1})$ for all $s \in S_\Gamma$ via the assignment $w(g,gs) := w_0(s)$ for all $g \in \Gamma$ and $s \in S^*_\Gamma$. A special choice of edge weights on $G$ is the trivial choice, given by $w_0 \equiv \mathbbm{1}$. we have the following curvature monotonicity results for weighted Cayley graphs:

\begin{thm}[Curvature monotonicity under addition of relators] \label{thm:addrel}
Let $\Gamma = \langle S\, |\, R \rangle$ and $\Gamma' = \langle S\, |\, R' \rangle$ be two finitely presented groups with $R' \supset R$ and $\Phi: G \to G'$ be the canonical associated $1$-Lipschitz map on their Cayley graphs $G = \Cay(\Gamma,S)$ and $G' = \Cay(\Gamma',S)$. Let $(m,w)$ and $(m',w')$ be weighting schemes on $G$ and $G'$ with trivial vertex measures $m,m'$ and edge weights associated to the functions $w_0: S^*_\Gamma \to (0,\infty)$ and $w_0': S^*_{\Gamma'} \to (0,\infty)$, respectively. 
Assume that $w_0$ and $w_0'$ satisfy the following for all $s' \in S^*_{\Gamma'}$:
\begin{equation} \label{eq:w0w0'}
    w_0'(s') = \sum_{s \in S^*_\Gamma: \Phi(s) = s'} w_0(s).
\end{equation}
Then the graphs $(G,m,w)$ and $(G',m',w')$ have constant weighted Bakry-\'Emery curvatures $K_G$ and $K_{G'}$ satisfying
$$ K_{G'} \ge K_G. $$
Moreover, for every edge $\{x',y'\}\in E'$ in $G'$ and every edge $\{x,y\}\in E$ in $G$ with $\{x',y'\} = \Phi(\{x,y\})$, the corresponding weighted Ollivier Ricci curvatures satisfy
$$ \kappa_{LLY}(x',y') \ge \kappa_{LLY}(x,y). $$
\end{thm}

Let us finish this subsection with two remarks about the above theorem: 
\begin{itemize}
\item[(a)] In the setting of the Theorem \ref{thm:addrel}, there exists, for every edge $\{x',y'\} \in E$ in $G'$, at least one edge in $G$ which is mapped to $\{x',y'\}$ under $\Phi$. Hence the sum in \eqref{eq:w0w0'} is non-empty. In the case $w_0 \equiv \mathbbm{1}$, the weights of all edges in $G$ are equal to $1$, and the adapted edge weights induced by $w_0'$ on $G'$ are also integer-valued, where $w_0'(s')$ counts the number of generators $s$ in $S^*_\Gamma$, which are merged into $s' \in S^*_{\Gamma'}$ under the map $\Phi$.
\item[(b)] The theorem confirms Conjecture 9.3 of \cite{CLP20} in the case of infinite dimension $\mathcal{N} = \infty$. The conjecture there was formulated carefully to avoid any merging and collapsing of edges. It is easy to verify that the proof of our theorem in Section 5 works equally well for finite dimension $\mathcal{N}  < \infty$. In this paper, however, we chose to not introduce Bakry-\'Emery curvature for finite dimension, in order to keep the presentation a bit simpler.   
\end{itemize}

\subsection{Outline}

In Section \ref{sec:basicsRAACHs} we discuss and prove some interesting combinatorial properties 
of RAACHs and their Cayley graphs. A detailed introduction into Ollivier Ricci and Bakry-\'Emery curvature is given in Section \ref{sec:curvatures}. The main curvature results for RAACHs (Theorems \ref{thm:mainraachor} and \ref{thm:mainraachbe}) are proved in Section \ref{sec:curvRAACHs}. Finally, monotonicity of weighted curvatures on Cayley graphs under addition of relators (Theorem \ref{thm:addrel}) is proved in Section \ref{sec:curvmon}. 

\section{Fundamental facts about RAACHs}
\label{sec:basicsRAACHs}

Proposition \ref{prop:elimR4Rinf} is a straightforward consequence of the following two propositions, which allow successive removal of all elements in the sets $R_4$ and $R_\infty$ of a RAACH $\Gamma$ without changing the combinatorics of its Cayley graph $\Cay(\Gamma,S)$.

\begin{prop} \label{prop:R4empty}
  Let $\Gamma$ be a RAACH with defining graph $(H,m)$, non-empty set $R_4$, and $H^*=(S^*,E^*)$ be the associated pair. Fix an element $s_0 \in R_4$ and define another defining graph $(H',m')$ with vertex set $S'$ as follows:
  \begin{itemize}
  \item $S' = S \setminus \{s_0\} \cup \{s',s''\}$;
  \item $m'(s) = \begin{cases} m(s) &\text{if $s \in S \setminus\{s_0\}$,} \\ 2 &\text{if $s \in \{s',s''\}$;} \end{cases}$
  \item all edges $\{s,t\}$ in $H$ with $s,t \neq s_0$ are also edges in $H'$, all edges $\{s_0,t\}$ in $H$ give rise to pairs of edges $\{s',t\}, \{s'',t\}$ in $H'$, and a new edge $\{s',s''\}$ is introduced in $H'$.
  \end{itemize}
  Let $\Gamma'$ be the RAACH to the new defining graph $(H',m')$ and $((H')^*,w')$ be the corresponding associated pair. Then there is an isomorphism between $(H^*,w)$ and $((H')^*,w')$ and between the Cayley graphs $G = \Cay(\Gamma,S)$ and $G' = \Cay(\Gamma',S')$.
\end{prop}

In short, any generator in $R_4$ can be replaced by a pair of commuting generators in $R_2$ without changing the combinatorics of the associated pairs and Cayley graphs.

\begin{proof}
  In the transition from $H$ to $H'$ the vertex $s_0 \in R_4 \subset S$ is replaced by the two new vertices $s',s'' \in R_2' \subset S'$. Since $s_0 \in R_4$, we have $w(s_0,s_0^{-1}) = 1$, and by introducing a new edge $\{s',s''\}$ in $H'$ we have also $w'(s',s'')=1$ and we require that $s'$ and $s''$ commute in $\Gamma'$. Moreover, every edge $\{s_0,t\}$ in $H$ gives rise to the edges $\{s_0^{\pm 1},t^{\pm 1}\}$ in $H^*$ and, by introducing the edges $\{s',t\}$ and $\{s'',t\}$ in $H'$, we also have the edges $\{s',t^{-1}\}$ and $\{s'',t^{-1}\}$ in $(H')^*$ if $t^{-1} \neq t$. All other edges are unaffected. These arguments show that we have a canonical graph isomorphism $H^* \to (H')^*$ by preserving $S^* \setminus \{s_0,s_0^{-1}\} = (S')^* \setminus \{s',s''\}$ and mapping $s_0$ and $s_0^{-1}$ in $S^*$ to $s'$ and $s''$ in $(S')^*$, respectively. Moreover, the edge weights $w$ and $w'$ agree under this isomorphism. 

  Next we describe the graph isomorphism $\Phi: \Gamma \to \Gamma'$ between the Cayley graphs $G$ and $G'$. For $\gamma \in \Gamma$, we choose a word $w$ with letters in $S^*$ representing $\gamma$. We then replace any power of $s_0$ (also negative ones) by a string of the simple letter $s_0$ representing the same power (for example, $s_0^{-1}$ is replaced by $s_0s_0s_0$). 
  Finally, by parsing the new word from left to right, we create a word $w'$ by alternatingly replacing any appearance of $s_0$ by $s'$ and $s''$ (ignoring all other letters in between) and starting with $s'$ from the left. Then $\Phi(\gamma)$ is the element in $\Gamma'$ represented by the word $w'$.  
  For example, if $\gamma = s_0 s_1 s_0^2 s_2^{-1}s_0$,
  we obtain $w'$ by the following process
  $$ w=s_0^{-1} s_1 s_0^2 s_2^{-1} s_0 \quad \mapsto \quad s_0 s_0 s_0 s_1 s_0 s_0 s_2^{-1} s_0 \quad \mapsto \quad s' s'' s' s_1 s'' s' s_2^{-1} s'' = w'. $$
  It is easy to see that all manipulations of $w$ via the powers and commutators in the presentation of $\Gamma$ can be mirrored by corresponding manipulations of $w'$ via the powers and commutators in the presentation of $\Gamma'$. This shows that the map $\Phi$ is well-defined. 
  
  The description of the inverse $\Phi^{-1}$ is a bit more complicated: Given $\gamma' \in \Gamma'$, choose a word $w'$ with letters in $(S')^*$ representing $\gamma'$. We first remove any even power of $s'$ and replace any odd power of $s'$ by $s'$ itself in $w'$, and we do the same with $s''$. In the second step we replace succcessively each occurence of $s',s''$ by $s_0,s_0^{-1}$ in a certain way by parsing through the word from left to right. If $s'$ appears first in the word, we replace it by $s_0$, if it is $s''$, we replace it by $s_0^{-1}$. Thereafter, we continue with replacements of $s',s''$ by $s_0,s_0^{-1}$ following this rule: in the case of two consecutive $s',s'$ or $s'',s''$ (potentially with other letters in between), we replace them by different powers of $s_0^{\pm 1}$, and in the case of $s',s''$ or $s'',s'$ we replace them by the same powers of $s_0^{\pm 1}$. Then $\Phi^{-1}(\gamma')$ is the element in $\Gamma$ represented by the resulting word $w$. Here is an example for that process:
  $$ w' = s_1 (s')^{-1} s_2 (s'')^2 s' s_1 s'' s' s_3 s' \quad \mapsto \quad s_1 s' s_2 s' s_1 s'' s' s_3 s' \quad \mapsto s_1 s_0 s_2 s_0^{-1} s_1 s_0^{-1} s_0^{-1} s_3 s_0 = w. $$
It is straighforward to verify that $\Phi^{-1} \circ \Phi = \Id_\Gamma$ and $\Phi \circ \Phi^{-1} = \Id_{\Gamma'}$. This shows that $\Phi$ is bijective. Moreover, by construction, $\Phi$ maps adjacent vertices in $G$ to adjacent vertices in $G'$ and is therefore a graph isomorphism. Note that $\Phi$ cannot a group isomorphism, since the groups $\Gamma$ and $\Gamma'$ are not isomorphic.
\end{proof}
    
\begin{prop} \label{prop:Rinfempty}
  Let $\Gamma$ be a RAACH with defining graph $(H,m)$, non-empty set $R_\infty$, and $H^*=(S^*,E^*)$ be the associated pair. Fix an element $s_0 \in R_\infty$ and define another defining graph $(H',m')$ with vertex set $S'$ as follows:
  \begin{itemize}
  \item $S' = S \setminus \{s_0\} \cup \{s',s''\}$;
  \item $m'(s) = \begin{cases} m(s) &\text{if $s \in S \setminus\{s_0\}$,} \\ 2 &\text{if $s \in \{s',s''\}$;} \end{cases}$
  \item all edges $\{s,t\}$ in $H$ with $s,t \neq s_0$ are also edges in $H'$, and all edges $\{s_0,t\}$ in $H$ give rise to pairs of edges $\{s',t\}, \{s'',t\}$ in $H'$.
  \end{itemize}
  Let $\Gamma'$ be the RAACH to the new defining graph $(H',m')$ and $((H')^*,w')$ be the corresponding associated pair. Then there is an isomorphism between $(H^*,w)$ and $((H')^*,w')$ and between the Cayley graphs $G = \Cay(\Gamma,S)$ and $G' = \Cay(\Gamma',S')$.
\end{prop}

In short, any generator in $R_\infty$ can be replaced by a pair of non-commuting generators in $R_2$ without changing the combinatorics of the associated pairs and Cayley graphs.

\begin{proof}
  The proof of this proposition is similar to the proof of the previous proposition. In the transition from $H$ to $H'$ the vertex $s_0 \in R_\infty \subset S$ is replaced by the two new vertices $s',s'' \in R_2' \subset S'$. Since $s_0 \in R_\infty$, we have $w(s_0,s_0^{-1}) = 0$, which matches  $w'(s',s'')=0$, meaning that $s'$ and $s''$ do not commute in $\Gamma'$. The arguments for all other edges are as in the previous proof and we conclude that 
  we have a canonical isomorphism $H^* \to (H')^*$ by preserving $S^* \setminus \{s_0,s_0^{-1}\} = (S')^* \setminus \{s',s''\}$ and mapping $s_0$ and $s_0^{-1}$ in $S^*$ to $s'$ and $s''$ in $(S')^*$, respectively. Moreover, the edge weights $w$ and $w'$ agree under this isomorphism. 

  Next we describe the map $\Phi: \Gamma \to \Gamma'$. For $\gamma \in \Gamma$, we choose a word $w$ with letters in $S^*$ representing $\gamma$. We then replace any positive power of $s_0$ by a string of the simple letter $s_0$ and any negative power of $s_0$ by a string of the letter $s_0^{-1}$.  
  In the second step we parse through the word from left to right, and if the first occurrence of one of the letters $s_0^{\pm 1}$ is
  $s_0$, we replace it by $s'$, and if it is $s_0^{-1}$, we replace it by $s''$. For later replacements to obtain a new word $w'$, we follow this rule: two consecutive $s_0,s_0$ or $s_0^{-1},s_0^{-1}$ (possibly with other letters in between) are replaced by an alternating pair $s',s''$ or $s'',s'$ (with the same letters in between), and two consecutive $s_0^{\pm 1},s_0^{\mp 1}$ with different powers (possibly with other letters in between) are replaced by the pair $s',s'$ or $s'',s''$ (note that these pairs do not need to cancel each other out since there may be other letters in between). Then $\Phi(\gamma)$ is the element represented by the word $w'$. Here is an example illustrating this replacement process:  
  $$ w = s_0^{-2} s_1 s_0^2 s_2 s_0 \quad \mapsto \quad s_0^{-1} s_0^{-1} s_1 s_0 s_0 s_2 s_0 \quad \mapsto \quad s''s's_1 s's''s_2 s' = w'. $$
  
  As in the previous proof, it is easy to see that all manipulations of $w$ via the powers and commutators in the presentation of $\Gamma$ can be mirrored by corresponding manipulations of $w'$ via the powers and commutators in the presentation of $\Gamma'$, showing that $\Phi$ is well-defined. 
  
  The description of the inverse $\Phi^{-1}$ is as follows: Given $\gamma' \in \Gamma'$, choose a word $w'$ with letters in $(S')^*$ representing $\gamma'$. We first remove any even power of $s'$ and replace any odd power of $s'$ by $s'$ itself in $w'$, and we do the same with $s''$. The second step starts with finding the first occurrence (from the left) of one of $s',s''$ in the word, and if it is $s'$, we replace it by $s_0$, and if it is $s''$, we replace it by $s_0^{-1}$. Then any pair of consecutive $s_0,s_0$ or $s_0^{-1},s_0^{-1}$ (possibly with letters in between) are replaced by alternating pairs $s',s''$ or $s'',s'$, and any pair of consecutive $s_0^{\pm 1}, s_0^{\mp 1}$ with different powers are replaced by pairs $s',s'$ or $s'',s''$. Then $\Phi^{-1}(\gamma')$ is the element in $\Gamma$ represented by the resulting word $w$. Let us again illustrate this replacement process by an example:
  $$ w' = (s')^{-1}s_1s's''s_2^2s'(s'')^2 \quad \mapsto s's_1s's''s_2^2s' \quad \mapsto \quad s_0s_1s_0^{-1}s_0^{-1}s_2^2s_0^{-1} = w. $$
  As in the previous proof, we conclude that $\Phi$ is bijective and maps adjacent vertices in $G$ to adjacent vertices in $G'$ and is therefore a graph isomorphism.
\end{proof}

In our next result, we describe short cycles in the Cayley graph of a RAACH.

\begin{prop} \label{prop:cycles}
  Let $\Gamma$ be a RAACH with defining graph $(H,m)$ and $S^* = \{s,s^{-1}\, \mid \ s \in S\}$. Then all $3$-, $4$- and $5$-cycles in $G=\Cay(\Gamma,S)$ have the following form:
  \begin{itemize}
      \item[(i)] Any $3$-cycle containing the identity $e \in \Gamma$ is of the form $s,s^2,s^3$ with $s^3=e$ for some $s \in S^*$ with $\ord{s}=3$.
      \item[(ii)] Any $4$-cycle containing $e$ is of the form $s,s^2,s^3,s^4$ with $s^4=e$ for some $s \in S^*$ with $\ord{s}=4$ or of the form $s,st,t,e$ for $s,t \in S^*$, $[s,t] = e$ and $t \neq s,s^{-1}$.
      \item[(iii)] For any $5$-cycle $$s_1, \,\, s_1s_2, \,\, s_1s_2s_3, \,\, s_1s_2s_3s_4,\,\, s_1s_2s_3s_4s_5 $$ containing $e$ with $s_1,s_2,s_3,s_4,s_5 \in S^*$ and $s_1s_2s_3s_4s_5=e$, there must exist two commuting elements $s,t \in S^*$ with $t \neq s,s^{-1}$ and $\ord(t) =3$, such that $\{s_1,s_2,s_3,s_4,s_5\} = \{s^{\pm 1},t^{\pm 1}\}$ and if $s_1 = s_5$, then $\ord(s_1)=3$.
 \end{itemize}
\end{prop}

\begin{proof}
    Let $\Gamma$ be a RAACH with defining graph $(H,m)$ and generators $S$. It follow from the nature of the relations that 
    \begin{equation} \label{eq:wordid}
    w = s_1 s_2 \dots s_n = e \end{equation}
    with $s_1,\dots,s_n \in S^*$ implies that the exponents of each generator $s \in S$ in the word $w$ must add up to an integer multiple of $m(s)$. Therefore, if we have for some $s \in S^*$ that $s$ appears in \eqref{eq:wordid}, then
    there must be at least two different $s_i,s_j$ with $1 \le i < j \le n$ with $s_i,s_j \in \{s,s^{-1}\}$. Therefore, any $3$-cycle $s_1,s_1s_2,s_1s_2s_3=e$ cannot involve $s^{\pm 1}, t ^{\pm}$ of two different generators $s,t \in S$, and it must be of the form described in (i).

    Similarly, for any $4$- or $5$-cycle $s_1,\dots,s_1s_2\cdots s_n=e$ with $n=4,5$ there can be at most two $s,t \in S^*$ with $s \neq t,t^{-1}$ such that $s_1,\dots,s_n \in \{s^{\pm 1}, t^{\pm 1}\}$.

    If a $4$-cycle involves only one $s \in S^*$, then it must be of the form $s,s^2,s^3,s^4=e$ or $s^{-1},s^{-2},s^{-3},s^{-4}=e$, since cycles to do allow multiple appearances of the same vertex. In this case we have $\ord(s) = \ord(s^{-1}) =4$.
    
    If a $4$-cycle involves $s,t \in S^*$, $s \neq t,t^{-1}$, possibly with their inverses, then we can assume, without loss of generality, that this $4$-cycle starts off either with $s,s^2$ or with $s,st$. In the first case, it must continue with $s,s^2,s^2t$ (after possibly renaming $t$ by $t^{-1}$), and the full $4$-cycle must be of the form $s,s^2,s^2t,s^2t^2$ ($s,s^2,s^2t,s^2$ and $s,s^2,s^2t,s^2ts^{\pm 1}$ are ruled out due to multiple occurrences of vertices or the fact that the last vertex is not the identity). Since we require $s^2t^2=e$, we conclude that $\ord(s)=\ord(t)=2$, which is a contradiction since then $s^2=e$ and the sequence $s,s^2=e,s^2t,s^2t^2=e$ would not be a cycle. Therefore, the $4$-cycle must start with $s,st$, and the next vertex must be either $st^2$ or $sts^{\pm 1}$. In the first case, the last vertex must be $st^2s^{\pm 1}$ ($st^3$ is ruled out since it cannot represent the identity), which would lead to $\ord(t)=2$ and, consequently, to a contradiction to the cycle condition since then $st^2 = s$. It remains to consider $4$-cycles starting with $s,st,sts^{\pm 1}$, which can only be completed by $s,st,sts^{\epsilon_1},sts^{\epsilon_1}t^{\epsilon_2}=e$ with $\epsilon_1,\epsilon_2 \in \{\pm 1\}$, since otherwise the last vertex could not be the identity.  
    The choice $\epsilon_1=1$ would lead to $\ord(s)=2$, in which case we would have $s=s^{-1}$. As similar argument applies to $t$. Therefore, the $4$-cycle involving $s,t \in S^*$ with $t \neq s,s^{-1}$ must be of the form $s,st,sts^{-1},sts^{-1}t^{-1}=e$, which implies that $s$ and $t$ commute. In this case the $4$-cycle simplifies to $s,st,t,e$, completing the proof of (ii). 

    For the proof of (iii), note first that $w=s_1s_2s_3s_4s_5$ cannot involve only one generator $s \in S^*$ and its inverse, since we can only have $\ord{s} \in \{2,3,4,\infty\}$, which rules this possibility out. Therefore, the word $w$ must involve two elements $s,t \in S^*$ with $t \neq s,s^{-1}$, that is $\{s_1,s_2,s_3,s_4,s_5\} \subset \{s^{\pm 1},t^{\pm 1} \}$. Since the exponents of $s$ and $t$ in $w$ must each add up to integer multiples of $\ord(s)$ and $\ord(t)$, respectively, we must have $2$ occurrences of $s^{\pm 1}$ and $3$ occurrences of $t^{\pm 1}$, or vice versa.
    Without loss of generality, assume that $t^{\pm 1}$ has $3$ occurrences in $w$. This forces the order of $t$ 
    to be $3$. By changing $s$ into $s^{-1}$, if needed, we can assume that $w$ contains the letters $s,s$ or the letters $s,s^{-1}$. In the first case, we must have $\ord(s)=2$, which means that $s=s^{-1}$, therefore, we can assume that the element represented by $w$ can also be expressed by a word, written again as $s_1s_2s_3s_4s_5$ with $3$ occurrences of $t$ (if there are $3$ occurrences of $t^{-1}$ instead, we can also replace $t$ by $t^{-1}$) as well as one occurrence of each $s$ and $s^{-1}$. Since $w$ comes from a $5$-cycle without multiple vertices and since we can cyclically shift the letters of the word $w$ such that $s$ comes to the front, we can conclude that we must have either the relation $sts^{-1}tt=e$ or $stts^{-1}t=e$. In the first case we obtain $[s^{-1},t^{-1}]=e$, and in the second case $[s^{-1},t]=e$, that is $s$ and $t$ must commute and the $5$-cycle can be written that all four elements $s,s^{-1},t,t^{-1}$ are involved. Moreover, the words $sts^{-1}tt$ and $stts^{-1}t$ and their cyclic shifts imply that the condition $s_1=s_5$ necessarily implies $s_1=t$, and therefore $\ord(s_1)=3$. This concludes the proof of (iii).
\end{proof}

We finish this section with the following fact, whose proof is straightforward and therefore omitted.

\begin{prop} Let $\Gamma_1$ and $\Gamma_2$ be two RAACHs with defining graphs $(H_1,m_1)$ and $(H_2,m_2)$, respectively. The direct product $\Gamma_1 \times \Gamma_2$ is again a RAACH with defining graph $(H,m)$, where $H$ is the disjoint union of $H_1$ and $H_2$ with additional edges between any pair of vertices in $H_1$ and $H_2$ and the vertex measure $m$, restricted to the vertices of $H_i$, $i=1,2$, coincides with the vertex measure $m_i$.
\end{prop}
    
    \section{Curvature notions} \label{sec:curvatures}

    In this section we provide a detailed introduction into Olliver Ricci curvature and Bakry-\'Emery curvature. 
       
    \subsection{Ollivier Ricci curvature}
Ollivier Ricci curvature was introduced in \cite{Oll09} and is based on Optimal Transport Theory. A fundamental concept in this theory is the
Wasserstein distance between probability measures.
\begin{defin}
Let $G = (V,E)$ be a graph. Let $\mu_{1},\mu_{2}: V \to [0,1]$ be two probability measures on $V$, that is, $\sum_{x \in V} \mu_i(x) = 1$ for $i=1,2$. The {\it Wasserstein distance} $W_1(\mu_{1},\mu_{2})$ between $\mu_{1}$ and $\mu_{2}$ is defined as
\begin{equation} \label{eq:W1def}
W_1(\mu_{1},\mu_{2})=\inf_{\pi} \sum_{y\in V}\sum_{x\in V} d(x,y)\pi(x,y),
\end{equation}
where the infimum runs over all transport plans $\pi:V\times  V\rightarrow [0,1]$ satisfying
$$\mu_{1}(x)=\sum_{y\in V}\pi(x,y),\:\:\:\mu_{2}(y)=\sum_{x\in V}\pi(x,y).$$
\end{defin}
The transport plan $\pi$ moves a mass
distribution given by $\mu_1$ into a mass distribution given by
$\mu_2$, and $W_1(\mu_1,\mu_2)$ is a measure for the minimal effort
which is required for such a transition.
If $\pi$ attains the infimum in \eqref{eq:W1def} we call it an {\it
  optimal transport plan} transporting $\mu_{1}$ to $\mu_{2}$. We define the following probability measures $\mu_x$ for any
$x\in V,\: p\in[0,1]$:
\begin{equation} \label{eq:muxp}
\mu_x^p(z)=\begin{cases}p,&\text{if $z = x$,}\\
\frac{1-p}{d_x},&\text{if $z\sim x$,}\\
0,& \mbox{otherwise.}\end{cases}
\end{equation}
These measures can be viewed as probabilistic realisations of $1$-balls around $x$ or, alternatively, as random walks with \emph{idleness} $p$.

\begin{defin}
The \emph{$p-$Ollivier-Ricci curvature} of an edge $\{x,y\} \in E$ in $G=(V,E)$ is
$$\kappa_{ p}(x,y)=1-W_1(\mu^{ p}_x,\mu^{ p}_y).$$
The Ollivier Ricci curvature introduced by Lin-Lu-Yau in
\cite{LLY11} is defined as
$$\kappa_{LLY}(x,y) = \lim_{ p\rightarrow 1}\frac{\kappa_{ p}(x,y)}{1- p}.$$
\end{defin}

A fundamental concept in the optimal transport theory is Kantorovich duality. First we recall the notion
of a 1--Lipschitz functions and then state Kantorovich duality.

\begin{defin}
Let $G=(V,E)$ be a combinatorial graph and $f: V\rightarrow\mathbb{R}$. We say that $f$ is \emph{$1$-Lipschitz} if 
$$|f(x) - f(y)| \leq d(x,y)$$
for all $x,y\in V.$ The set of all $1$--Lipschitz functions is denoted by \textrm{1--Lip}. 
\end{defin}

\begin{thm}[Kantorovich duality]\label{Kantorovich}
Let $G = (V,E)$ be a combinatorial graph and $\mu_{1},\mu_{2}$ be two probability measures on $V$. Then
$$W_1(\mu_{1},\mu_{2})=\sup_{\substack{\phi:V\rightarrow \mathbb{R}\\ \phi\in \textrm{\rm{1}--{\rm Lip}} 
}}  \sum_{x\in V}\phi(x)(\mu_{1}(x)-\mu_{2}(x)).$$
If $\phi \in \textrm{1--Lip}$ attains the supremum we call it an \emph{optimal Kantorovich potential} transporting $\mu_{1}$ to $\mu_{2}$.
\end{thm}

The following result 
allows us to use a convenient choice of idleness parameter $p$ 
to compute $\kappa_{LLY}(x,y)$  from $\kappa_{p}(x,y)$. 

\begin{thm}[see \cite{BCLMP18}]\label{LLY_vs_p-OR}
  Let $G=(V,E)$ be a combinatorial graph. Let $x,y\in V$ with
  $x\sim y.$ Then the function $p \mapsto \kappa_{p}(x,y)$ is concave
  and piecewise linear over $[0,1]$ with at most $3$ linear
  parts. Furthermore $\kappa_{p}(x,y)$ is linear on the intervals
  \begin{equation*}
    \left[0,\frac{1}{{\rm{lcm}}({\rm{deg}}(x),{\rm{deg}}(y))+1}\right]\:\:\: {\rm and} \:\:\:\left[\frac{1}{\max({\rm{deg}}(x),{\rm{deg}}(y))+1},1\right].
  \end{equation*}
  Thus, if $p\in \left[\frac{1}{\max({\rm{deg}}(x),{\rm{deg}}(y))+1},1\right]$ then 
  \begin{equation} \label{eq:relLLykp}
  \kappa_{LLY}(x,y) = \frac{1}{1-p}\kappa_{p}(x,y).
  \end{equation}
\end{thm}

The probability measures \eqref{eq:muxp} give rise to a corresponding \emph{random walk Laplacian}, defined on functions $f: V \to \mathbb{R}$ as follows:
$$ \Delta f(x) = \sum_{y \in V} \mu_x^p(y) (f(y)-f(x)). $$
We can choose other graph Laplacians $\Delta$ which are no longer based on probability measures. In fact, the definition 
$$ \kappa_{LLY}(x,y) = \inf_{\substack{f \in \textrm{\rm{1}--{\rm Lip}}\\ f(y)-f(x) = 1}} \Delta f(x) - \Delta f(y) $$
in \cite{MW19} recovers the original Ollivier Ricci curvature for random walk Laplacians and 
leads to a generalization of Ollivier Ricci curvature for more general Laplacians. We use this more general viewpoint for the results in Theorem \ref{thm:addrel}.

\subsection{ Bakry-\'Emery curvature }

Bakry-\'Emery introduced in \cite{BE84} a curvature notion which was based on two symmetric bilinear $\Gamma$-operators involving the 
generator (Laplacian) of a continuous time Markov process. 
In the setting of a combinatorial graph $G=(V,E)$, these $\Gamma$-operators are defined via a graph Laplacian $\Delta$ for a pair of functions $f,g: V \to \mathbb{R}$ as follows:
\begin{eqnarray*}
    {\bf{\Gamma}}(f,g) &=& \frac{1}{2}\left( \Delta(fg) - f\Delta g - g \Delta f\right), \\
    {\bf{\Gamma}}_2(f,g) &=& \frac{1}{2}\left( \Delta {\bf{\Gamma}}(f,g) - {\bf{\Gamma}}(f,\Delta g) - {\bf{\Gamma}}(g,\Delta f)\right).
\end{eqnarray*}
In the case of an unweighted combinatorial graph, the Laplacian for Bakry\'Emery curvature used in the paper is the non-normalized Laplacian given by
\begin{equation} \label{eq:nonormlap}
\Delta f(x) = \sum_{y \sim x} (f(y)-f(x)). 
\end{equation}
We write ${\bf{\Gamma}} f$ for ${\bf{\Gamma}}(f,f)$ and, similarly, ${\bf{\Gamma}}_2 f$
for ${\bf{\Gamma}}_2(f,f)$.
Usually, the Bakry-\'Emery curvature involves a dimension parameter $N \in (0,\infty]$, but we restrict our consideration to the case $N=\infty$. In the graph theoretical setting, Bakry-\'Emery curvature is a function $K: V \to \mathbb{R}$ on the vertices.

\begin{defin}
    The \emph{Bakry-\'Emery curvature} $K(x)$ (for dimension $N=\infty$) of a vertex $x \in V$ in $G=(V,E)$ is the supremum of all values $K \in \mathbb{R}$ such that
    \begin{equation} \label{eq:BKcurvdef} 
    {\bf{\Gamma}}_2 f(x) - K {\bf{\Gamma}} f(x) \ge 0 \quad \text{for all $f: V \to \mathbb{R}$.} 
    \end{equation}
\end{defin}

Let us reformulate the condition \eqref{eq:BKcurvdef} in matrix form. Firstly, ${\bf{\Gamma}} f(x)$ and ${\bf{\Gamma}}_2 f(x)$ do not change by adding a constant to the function $f$, so we can restrict the condition \eqref{eq:BKcurvdef} to functions $f: V \to \mathbb{R}$ 
with $f(x) = 0$. Next, we employ the fact that the left hand side of \eqref{eq:BKcurvdef} does only depend on the values of $f$ in the $2$-ball around $x$: Assume that there are $m = d_x$ vertices at distance $1$ from $x$ and $n$ vertices as distance $2$ from $x$ and denote these vertices by $y_1,\dots,y_n$ and $z_1,\dots,z_n$, respectively. Let $$ \vec{f}_1(x) = (f(y_1),\dots,f(y_n)) \quad \text{and} \quad \vec{f}_2(x) = (f(z_1),\dots,f(z_n)). $$
Then the inequality in \eqref{eq:BKcurvdef} can be reformulated with the help of a suitable symmetric $n \times n$ matrix ${\bf{\Gamma}}(x)$ and of a suitable symmetric $(n+m) \times (n+m)$ matrix ${\bf{\Gamma}}_2(x)$ as
$$ (\vec{f}_1,\vec{f}_2) {\bf{\Gamma}}_2(x) \begin{pmatrix} \vec{f}_1 \\ \vec{f}_2 \end{pmatrix} - \vec{f}_1 (K {\bf{\Gamma}}(x)) \vec{f}_1 \ge 0. $$
This reduces the determination of $K(x)$ as a semidefinite programming problem.

Writing ${\bf{\Gamma}}_2(x)$ as a block matrix with blocks of size $n$ and $m$,
$$ {\bf{\Gamma}}_2(x) = \begin{pmatrix} ({\bf{\Gamma}}_2(x) )_{S_1,S_1} &
({\bf{\Gamma}}_2(x) )_{S_1,S_1} \\
({\bf{\Gamma}}_2(x) )_{S_2,S_1} &
({\bf{\Gamma}}_2(x) )_{S_2,S_2} \end{pmatrix}, $$
we can reformulate \eqref{eq:BKcurvdef} further as
\begin{equation} \label{eq:BKcurvdef2} 
\begin{pmatrix} ({\bf{\Gamma}}_2(x) )_{S_1,S_1} - K {\bf{\Gamma}}(x) &
({\bf{\Gamma}}_2(x) )_{S_1,S_1} \\
({\bf{\Gamma}}_2(x) )_{S_2,S_1} &
({\bf{\Gamma}}_2(x) )_{S_2,S_2} \end{pmatrix} \succcurlyeq 0, 
\end{equation}
where $A \succcurlyeq B$ and $A \succ B$ for symmetric matrices mean that $A-B$ is positive semidefinite and that $A-B$ is positived definite, respectively. Now we employ the Schur complement for symmetric block matrices 
$$ M = \begin{pmatrix} M_{11} & M_{12} \\ M_{21} & M_{22} \end{pmatrix} $$
with invertible $M_2$, defined by
$$ M/M_{22} := M_{11} - M_{12} M_{22}^{-1} M_{21}. $$
Using the fact that $M \succcurlyeq 0$ if and only if $M_2 \succcurlyeq 0$, \eqref{eq:BKcurvdef2} is equivalent to
\begin{equation} \label{eq:BKcurvdef3} 
Q(x) - K {\bf{\Gamma}}(x) \succcurlyeq 0 
\end{equation}
where
$$ Q(x) = {\bf{\Gamma}}_2(x)/\left ({\bf{\Gamma}}_2(x)\right)_{S_2,S_2}. $$
All derivations so far were carried out with respect to the non-normalized Laplacian \eqref{eq:nonormlap}, but they are equally valid in the case of the weighted Laplacian \eqref{eq:wlap}.
The non-normalized Laplacian can be viewed as the weighted Laplacian with trivial vertex measure $m = \mathbbm{1}_V$ and trivial edge weights $w = \mathbbm{1}_E$. Introducing finally the symmetric matrix
$$ A(x) := 2 D(x)^{-1} Q(x) D(x) \quad \text{with} \,\, D(x) = \mathrm{diag}(\sqrt{w(x,y_1)}, \dots, \sqrt{w(x,y_n)}), $$
where $\mathrm{diag}(c_1,\dots,c_n)$ denotes the diagonal matrix with entries $c_1,\dots,c_n$ on its diagonal, 
\eqref{eq:BKcurvdef3} is equivalent to $A(x) - K {\rm{Id}}_n) \succcurlyeq 0$. It follows from these manipulations that the determination of the Bakry-\'Emery curvature $K(x)$ reduces to finding the smallest eigenvalue of the matrix $A(x)$, that is
$$ K(x) = \lambda_{\rm{min}}(A(x)). $$
The references for this reformulation process are \cite{Sic20,Sic21} and \cite{CKLP22}. We will make use of the following description of the curvature matrix:
\begin{equation} \label{eq:AxLapl} 
A(x) = - 2\Delta_{S_1(x)} -2\Delta_{S_1'(x)} + J + \frac{3-D}{2} {\rm{Id}} - \frac{1}{2}{\rm{diag}}(d_1^+,\dots,d_D^+). 
\end{equation}
Here $J$ is the $D \times D$ all-one matrix, $\Delta_{S_1(x)}$ is the Laplacian of the subgraph induced by $S_1(x)$ with weights $w_{ij}=1$ if and only if vertices $y_i$ and $y_j$ are neighbours in this induced subgraph, $\Delta_{S_1'(x)}$ is the Laplacian of the weighted graph with vertex set $S_1(x)$, vertex measure $m \equiv 1$, and edge weights $w_{ij}^{S_1'(x)} = \sum_{z \in S_2(x)} \frac{w_{y_iz}w_{y_jz}}{d_z^-}$ for $i \neq j$ and $0$ otherwise, where $w_{yz}=1$ if $z \in S_2(x)$ is adjacent to vertex $y \in S_1(x)$ and $0$ otherwise, and $d_z^-$ is the in-degree of $z \in S_2(x)$, that is, the number of vertices in $S_1(x)$ adjacent to $z$. Note that the matrices representing Laplacians have the property that their rows sums are all equal to zero. 

\section{Curvatures of RAACHs}
\label{sec:curvRAACHs}

In the next two subsections, we prove the curvature results for the Cayley graphs of certain RAACHs stated in the Introduction (Theorems \ref{thm:mainraachor} and \ref{thm:mainraachbe}). Due to vertex transitivity, it suffices to consider Ollivier Ricci curvatures $\kappa_{LLY}(s)$ of generators in $s \in S^*$ (which represent the edges incident to $e \in \Gamma$) and the Bakry-\'Emery curvature $K(e)$ at $e \in \Gamma$.

\subsection{Ollivier Ricci curvature of RAACHs}

We start with the following general fact about Ollivier Ricci curvature.

\begin{lemma} \label{lem:ollgeneral}
    Let $G=(V,E)$ be a combinatorial graph with two adjacent vertices $x,y \in V$. Assume that the neighbourhood structures of $x$ and $y$ are as illustrated in Figure \ref{OR_figure}, with or without an additional common vertex $z$ (illustrated in blue), that is, the degrees of $x$ and $y$ are equal and either $n+\ell+1$ or $n+\ell+2$ (depending on whether $z \not\in V$ or $z \in V$). Assume further that we have the following distances between these vertices in $G$:
    \begin{itemize}
        \item[(i)] $d_G(x_i,y_j) = 3$ for $i,j \in [\ell]$,
        \item[(ii)] $d_G(x_i,v_j) = d_G(y_i,u_j) = 3$ for $i \in [\ell]$ and $j \in [n]$,
        \item[(iii)] $d_G(u_j,v_j) = 1$ for $j \in [n]$.
    \end{itemize}
    Then we have
    $$ \kappa_{LLY}(x,y) = \begin{cases} \frac{2-2\ell}{n+\ell+1}, & \text{if $z \not\in V$,} \\ \frac{3-2\ell}{n+\ell+2}, & \text{if $z \in V$.}  \end{cases}$$
\end{lemma}

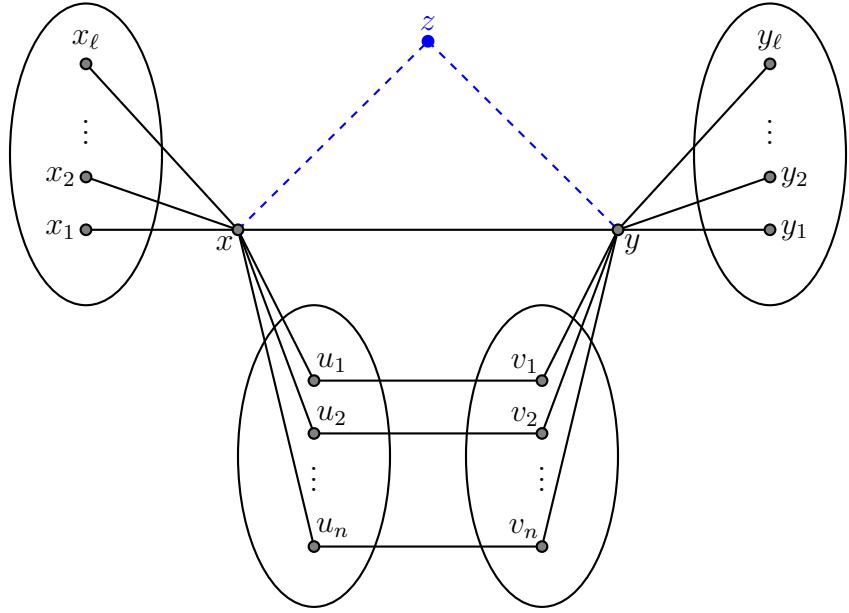
\begin{figure}[h]
\begin{center}
\tikzstyle{every node}=[circle, draw, fill=black!50,
                        inner sep=0pt, minimum width=4pt]      
\begin{tikzpicture}[thick,scale=1]%
    \draw (0,0) -- (5,0);

    \draw [dashed,color=blue](0,0) -- (2.5,2.5);
    \draw (2.5,2.5) node[color=blue,label=above:{\blue{$z$}}]{};
    \draw [dashed,color=blue](2.5,2.5) -- (5,0);
    
    \draw (1,-2) -- (4,-2);
     \draw (1,-2.7) -- (4,-2.7);
     \draw (1,-4.2) -- (4,-4.2);
         
    \draw (0,0)--(1,-2);
    \draw (0,0)--(1,-2.7);
    \draw (0,0)--(1,-4.2);
    \draw (1,-2) node[label=above right:$u_1$]{};
    \draw (1,-2.7) node[label=above right:$u_2$]{};
    \draw (1,-4.2) node[label=above right:$u_n$]{};
    \draw (1,-3) ellipse (1cm and 2cm);
    \node[draw=none,fill=none] at (1,-3.2) {\vdots};
    
    \draw (5,0)--(4,-2);
    \draw (5,0)--(4,-2.7);
    \draw (5,0)--(4,-4.2);
    \draw (4,-2) node[label=above left:$v_1$]{};
    \draw (4,-2.7) node[label=above left:$v_2$]{};
    \draw (4,-4.2) node[label=above left:$v_n$]{};
    \draw (4,-3) ellipse (1cm and 2cm);
    \node[draw=none,fill=none] at (4,-3.2) {\vdots};

    \draw (0,0)--(-2,0);
    \draw (0,0)--(-2,0.7);
    \draw (0,0)--(-2,2.2);
    \draw (-2,0) node[label=left:$x_1$]{};
    \draw (-2,0.7) node[label=left:$x_2$]{};
    \draw (-2,2.2) node[label={[label distance=0px]90:$x_\ell$}]{};
    \draw (-2,1) ellipse (1cm and 2cm);
    \node[draw=none,fill=none] at (-2,1.4) {\vdots};
     
    \draw (5,0)--(7,0);
    \draw (5,0)--(7,0.7);
    \draw (5,0)--(7,2.2);
    \draw (7,0) node[label= right:$y_1$]{};
    \draw (7,0.7) node[label= right:$y_2$]{};
    \draw (7,2.2) node[label={[label distance=-0px]90:$y_{\ell}$}]{};
    \draw (7,1) ellipse (1cm and 2cm);
    \node[draw=none,fill=none] at (7,1.4) {\vdots};
\draw (0,0) node[label=below left:$x$]{};
    \draw (5,0) node[label=below right:$y$]{};

    \end{tikzpicture}
\end{center}
\caption{The labelling of the vertices in the neighbourhoods of $x$ and $y$.}
\label{OR_figure}
\end{figure}

\begin{proof}
    Let us fist consider the case $z \not\in V$ and set $D = {\rm{deg}}(x)={\rm{deg}}(y)=n+\ell+1$ and $p=1/(D+1)$. Consider the transport plan transporting $\mu_x^p$ to $\mu_y^p$ given by $\pi(x_i,y_i)= 1/(D+1)$ for all $i \in [\ell]$ and $\pi(u_j,v_j) = 1/(D+1)$ for all $j \in [n]$ and $\pi(w_1,w_2) = 0$ for all other choices of pairs of vertices.  Then we have
    $$ W_1(\mu_x^p,\mu_y^p) \le \frac{3\ell + n}{D+1}. $$
    Let 
    $$ V_0 := \{x_1,\dots,x_\ell,x,y,y_1,\dots,y_\ell,u_1,\dots,u_n,v_1,\dots,v_n\} , $$
    and $\phi_0: V_0 \to \R$ be a 1-Lip function given by
    $$ \phi_0(w) = \begin{cases} 3, & \text{if $w = x_i$ for $i \in [\ell]$}, \\
    2, &\text{if $w=x$ or $w=u_j$ for $j \in [n]$}, \\ 1, & \text{if $w=y$ or $w=v_j$ for $j \in [n]$,} \\ 
    0, & \text{if $w = y_i$ for $i \in [\ell]$.} \end{cases} $$
    This function can be extended to a 1-Lip function $\phi: V \to \R$ 
    via (see 
    \cite{McS34} or Kirszbraun Theorem)
    $$ \phi(w) = \sup \{ \phi_0(w_0) - d(w_0,w) \, \mid \, w_0 \in V_0 \}, $$
    and Kantorovich duality (Theorem \ref{Kantorovich}) yields
    $$ W_1(\mu_x^p,\mu_y^p) \ge \sum_{w \in V} \phi(w)(\mu_x^p(w)-\mu_y^p(w)) = \frac{3\ell+n}{D+1}. $$
    Consequently, we obtain using \eqref{eq:relLLykp},
    $$ \kappa_{LLY}(x,y) = \frac{D+1}{D} \kappa_p(x,y) = \frac{D+1-(3\ell+n)}{D} = \frac{2-2\ell}{n+\ell+1}. $$
    
    The case $z \in V$ is treated similarly with $D = n + \ell +2$, the same transport plan and the same function $\phi_0$. We then obtain $W_1(\mu_x^p,\mu_y^p) = \frac{3\ell+n}{D+1}$ and
    $$ \kappa_{LLY}(x,y) = \frac{D+1-(3\ell+n)}{D} = \frac{3-2\ell}{n+\ell+2}. $$
\end{proof}

Now we present the proof of the Ollivier Ricci curvature result for RAACHs from the Introduction.

\begin{proof}[Proof of Theorem \ref{thm:mainraachor}]
    Let $\Gamma$ be a RAACH with generating set $S$, defining graph $(H,m)$ and associated pair $(H^*,w)$. Let $s \in S^*$.
    
    Assume first that $\ord(s) \neq 3,4$. Let $G = \Cay(\Gamma,S)$. Then we have the situation illustrated in Figure \ref{OR_figure} (without the vertex $z$) with $x=e$, $y=s$ and $u_1,\dots,u_n$ the elements in $S^* \setminus \{s,s^{-1}\}$ which commute with $s$ and $n = {\rm{deg}}_{H^*}(s)$, $x_1 = s^{-1}$, and $x_2,\dots,x_\ell$ the elements in $S^*$ which do not commute with $s$. Similarly, $v_1,\dots,v_n$ are the elements $su_j$ for $j \in [n]$ and $y_1=s^2$, $y_i = sx_i$ for $i \in [\ell] \setminus \{1\}$. The distance condition $d(x_i,y_j)=3$ of Lemma \ref{lem:ollgeneral} is satisfied since $s,y_j$ cannot be completed to a $3$-, $4$- or $5$-cycle containing $x_i$ because of Proposition \ref{prop:cycles} and since $x_1=s^{-1}$ (which would imply $\ord(s)=3$ in the case that $s,x_1$ and $e$ would be contained in a $5$-cycle) and $x_i$ does not commute with $s$ for $i \in [\ell]\setminus \{1\}$. Similarly, the distance condition $d(x_i,v_j) = 3$ is satisfied since $u_1,v_1=su_1$ cannot be completed to a $3$-, $4$- or $5$-cycle containing $x_i$ for the same reason. A similar argument shows $d(y_i,u_j) =3$. So we can apply Lemma \ref{lem:ollgeneral} with $n = {\rm{deg}}_{H^*}(s)$ and $n+\ell+1 = \deg_G(e)$, and we obtain
    \begin{equation} \label{eq:ollRAACHcurvno3} 
    \kappa_{LLY}(s) = \frac{2-2\ell}{n+\ell+1} = \frac{4+2 \deg_{H^*}(s)}{\deg_G(e)} - 2. 
    \end{equation}
    
    In the case $\ord(s)=4$, the arguments are almost the same, with the only difference that $s^{-1}$ belongs now to the set $\{u_1,\dots,u_n\}$ and $s^2$ to the set $\{v_1,\dots,v_n\}$ (instead of $x_1=s$ and $y_1=s^2$). This leads to the same end result \eqref{eq:ollRAACHcurvno3}.

    In the case $\ord(s)=3$, 
    we have the situation illustrated in Figure \ref{OR_figure} (with the vertex $z$) with $x=e$, $y=s$, $z=s^2$ and $u_1,\dots,u_n$ the elements in $S^* \setminus \{s,s^{-1}\}$ which commute with $s$ and $x_1,\dots,x_\ell$ the elements in $S^*$ which do not commute with $s$. Again, we set $v_j = su_j$ for $j \in [n]$ and $y_i = sx_i$ for $i \in [\ell]$. The distance conditions of Lemma \ref{lem:ollgeneral} are verified via the same arguments, and we obtain 
    $$ \kappa_{LLY}(s) = \frac{3-2\ell}{n+\ell+2} = \frac{3+2 \deg_{H^*}(s)}{\deg_G(e)} - 2.  $$

    Finally, if we have $R_2=S$, then $H^*$ agrees with $H$, $S^*$ agrees with $S$, and we have no generators in $S^*$ of order $3$, 
    and \eqref{eq:ollRAACHcurvno3} simplifies to 
    $$ \kappa_{LLY}(s) = \frac{4+2 \deg_H(s)}{\deg_G(e)} - 2. $$
\end{proof}

\subsection{Bakry-\'Emery curvatures of certain RAACHs}

Our next aim is to derive the Bakry-\'Emery curvature matrix $A(e)$ of the Cayley graph of a RAACH $\Gamma$ with generating set $S$. A first step into this direction is the following lemma. 

\begin{lemma} \label{lem:assocpairsdetinc2balls} The associated pair $(H^*,w)$ of a RAACH $\Gamma$ with generating set $S$ determines the combinatorial structure of the incomplete $2$-ball $\mathring{B}_2(e)$ of the Cayley graph ${\rm{Cay}}(\Gamma,S)$. Moreover, we have the following identity between Laplacians:
\begin{equation} \label{eq:Lapidentity} 
\Delta_{H^*} = 2 \Delta_{S_1(e)} + 2 \Delta_{S_1'(e)} 
\end{equation}
with the Laplacians on the right hand side defined in \eqref{eq:AxLapl}.
\end{lemma}

\begin{proof}
The vertices in $S_1(e)$ coincide with $S^*$, which is the vertex set of $H^*$. To understand $\mathring{B}_2(e)$, we need to understand the $3$-cycles and $4$-cycles containing $e$. They are described in Proposition \ref{prop:cycles}
and are of the form  $e,s,s^{-1}$ for $s \in R_3 \subset S$, $e,s,s^2,s^3$ for $s \in R_4 \subset S$, and $e,s,st,t$ for any commuting pair of $s,t \in S^*$ with $s \neq t,t^{-1}$. This shows that any edge of ${\rm{Cay}}(\Gamma,S)$ connecting two different vertices in $S_1(e)$ is precisely described by vertices $s,t \in S^*$ with $w(s,t) = 2$ (and therefore $t=s^2$, $e=s^3$) and the corresponding off-diagonal entries of $\Delta_{H^*}$ and $2 \Delta_{S_1(e)} = 2 \Delta_{S_1(e)} + 2 \Delta_{S_1'(e)}$ coincide. Moreover, any $4$-cycle containing $e$ must stem from a commuting pair of different generators $s,t \in S^*$ satisfying $s \neq t^{-1}$ unless ${\rm{ord}}(s) = 4$. This is precisely described by the property $w(s,t) = 1$ and the corresponding off-diagonal of $\Delta_{H^*}$ agrees with the corresponding off-diagonal matrix of $2\Delta_{S_1'(e)}$, since in this case $s$ and $t$ have a unique common neighbour in $S_2(e)$, namely $st$, whose in-degree is equal to $2$. Moreover, the corresponding off-diagonal of $\Delta_{S_1(e)}$ is zero. This shows that all off-diagonal entries of $\Delta_{H^*}$ and $2 \Delta_{S_1(e)} + 2 \Delta_{S_1'(e)}$ agree and, since the off-diagonal entries uniquely determine the diagonal entries of Laplacians, we have the agreement \eqref{eq:Lapidentity} stated in the lemma. 
\end{proof}

We need the following lemma.

\begin{lemma}
  Let $\Gamma$ be a RAACH with generating set $S$ and defining graph $(H,m)$, $(H^*,w)$ be its associated pair, $S^*=\{s,s^{-1} \mid s \in S\}$ and $D=|S^*|$. If $R_3 = \emptyset$ and $D \ge 2$ or $R_3 = S$ and $D\ge 4$, we have
  $$ \lambda_2(-\Delta_{H^*}) \le |S^*|. $$
\end{lemma}

\begin{proof}
    We have the following variational characterization of $\lambda_2(-\Delta_{H^*})$:
    $$ \lambda_2(-\Delta_{H^*}) = \inf\left\{ \frac{\frac{1}{2} \sum_{s,s' \in S^*} w(s,s') (f(s')-f(s))^2}{\sum_{s \in S^*} (f(s))^2} \, \Big\vert \, \sum_{s \in S^*} f(s) = 0 \right\}. $$
    In the case $R_3=\emptyset$ with $D \ge 2$, we can choose $t_1,t_2 \in S^*$, $t_1 \neq t_2$, such that $f(t_1) = -f(t_2) = 1$ and $f(s)=0$ for all $s \in S^* \setminus \{t_1,t_2\}$. In the case $R_3=S$ with $D \ge 4$, we can choose $s_1,s_2 \in S$, $s_1 \neq s_2$, such that $f(s_1)=f(s_1^{-1}) = -f(s_2)=-f(s_2^{-1}) = 1$ and $f(s) = f(s^{-1})=0$ for all $s \in S \setminus \{s_1,s_2\}$, to satisfy $\sum_{s \in S^*} f(s) = 0$. Note that this choice implies that $w(s,s') \in \{0,1\}$ for all $s,s' \in S^*$ with $f(s')-f(s) = 0$, and we can use the estimate
    $$ \lambda_2(-\Delta_{H^*}) \le \frac{\frac{1}{2} \sum_{s,s' \in S^*} (f(s')-f(s))^2}{\sum_{s \in S^*} (f(s))^2}. $$
    In the case $R_3 = \emptyset$, this estimate yields
    $$ \lambda_2(-\Delta_{H^*}) \le \frac{4 + 4 + 4(D-2)}{4} = D. $$
    Similarly, in the case $R_3 = S$, we obtain
    $$ \lambda_2(-\Delta_{H^*}) \le \frac{8\cdot 4+ 8(D-4)}{8} = D. $$
\end{proof}

Now we can provide the proof of the Bakry-\'Emery curvature result for certain RAACHs from the Introduction.

\begin{proof}[Proof of Theorem \ref{thm:mainraachbe}]
    Using \eqref{eq:AxLapl} for the curvature matrix $A(e)$ and the identity \eqref{eq:Lapidentity}, we obtain
    $$ A(e) = - \Delta_{H^*} + J + \frac{3-D}{2}{\rm{Id}} - \frac{1}{2}{\rm{diag}}(d_1^*,\dots,d_D^+). $$
    Let $S_1(e)$ be the $1$-sphere of $e$ in the Cayley graph $G$.
    Since a vertex $s \in S_1(e)$ is adjacent to at most one other vertex in $S_1(e)$, in which case we have ${\rm{ord}}(s)=3$, we conclude that 
    $$ {\rm{diag}}(d_1^*,\dots,d_D^+) = (\underbrace{D-2,\dots,D_2}_{\ell},\underbrace{D-1,\dots,D-1}_{D-\ell}). $$
    Combining these facts, we end up with
    $$ A(e) = (2-D) {\rm{Id}} + J - \Delta_{H^*} + \frac{1}{2} {\rm{diag}}(\underbrace{1,\dots,1}_{\ell},\underbrace{0,\dots,0}_{D-\ell}). $$

    Since $K(e) = \lambda_{\rm{min}}(A(e))$, we conclude in the case $R_3 = \emptyset$ that
    $$ K(e) = 2-D + \lambda_{\rm{min}}(J - \Delta_{H^*}). $$
    The assumption $R_3 = \emptyset$ implies also that $(H^*,w)$ does not have any edges of weight $2$ and $\Delta_{H^*}$ is therefore the Laplacian of the combinatorial graph $H^*$.
    
    Note that $J$ and $-\Delta_{H^*}$ have a common orthogonal eigenvector decomposition: the eigenvalue of $J-\Delta_{H^*}$ to the constant eigenvector is $D$ and the eigenvalues of all other eigenvectors (perpendicular to the constant eigenvector) are $\lambda_j(-\Delta_{H^*})$ for $j \in [D] \setminus \{1\}$. Since
    $$ \lambda_2(-\Delta_{H^*}) \le \lambda_2(-\Delta_{K_D}) = D, $$
    (where $K_D$ denotes the complete graph with $D$ vertices), we obtain
    $$ K(e) = 2-D + \lambda_2(-\Delta_{H^*}). $$

    Similarly, we conclude in the case $R_3=S$ that
    $$ K(e) = \frac{5}{2} - D + \lambda_{\rm{min}}(J-\Delta_{H^*}). $$

    The eigenvalue analysis of $J-\Delta_{H^*}$ remains the same and we still have $\lambda_2(-\Delta_{H^*}) \le D$. Therefore, we obtain
    $$ K(e) = \frac{5}{2} -D + \lambda_2(-\Delta_{H^*}). $$
\end{proof}

\section{Adding relators does not decrease curvatures}
\label{sec:curvmon}

Before we present the proof of Theorem \ref{thm:addrel} in the Introduction, we first prove a more general result about surjective $1$-Lipschitz maps. A map $\Phi: V \to V'$ between the vertex sets of two combinatorial graphs $G=(V,E)$ and $G'=(V',E')$ is called a surjective $1$-Lipschitz map if we have $\Phi(V) = V'$ and
$$ d'(\Phi(x),\Phi(y)) \le d(x,y), $$
where $d$ and $d'$ are the combinatorial distance functions of the graphs $G$ and $G'$. Note that $\Phi$ does not necessarily induces a map between the edge sets $E$ and $E'$ (since two adjacent vertices $x,y\in V$ in $G$ can be mapped to the same vertex $\Phi(x)=\Phi(y)$ in $G'$). The following relation between weigthed Laplacians of surjective $1$-Lipschitz maps is important for our curvature results.

\begin{lemma} \label{lem:laprel}
  Let $\Phi: V \to V'$ be a surjective $1$-Lipschitz map between $G=(V,E)$ and $G'=(V',E')$. Let $(m,w)$ and $(m'w')$ be weigthing schemes on $G$ and $G'$, respectively. We assume the following relation between them:
  \begin{itemize}
      \item[(i)] For all $x' \in V'$ and
      $x \in \Phi^{-1}(x)$:
      \begin{equation} \label{eq:mm'rel} 
      m(x) = m'(x').
      \end{equation}
      \item[(ii)] For all $x',y' \in V'$ with $x' \sim y'$ and all $x \in \Phi^{-1}(x')$:
      \begin{equation} \label{eq:w'wrel} 
      w'(x',y') = \sum_{y \in \Phi^{-1}(y')} w(x,y). 
      \end{equation}
  \end{itemize}    
  Then we have the following relation between the weighted Laplacians $\Delta$ and $\Delta'$ on $G$ and $G'$:
  $$ (\Delta (f' \circ \Phi)(x)) = (\Delta' f')(\Phi(x))  $$
  for all functions $f': V' \to \mathbb{R}$ and $x \in V$.
\end{lemma}

\begin{proof}
    Let $x' = \Phi(x)$ and $f = f' \circ \Phi: V \to \mathbb{R}$. 
    Since $G$ is simple, we can extend our edge weights $w: E \to (0,\infty)$ to a symmetric map $w: V \times V \to [0,\infty)$ with $w(x,y) = 0$
    for $x=y$ and for $x \neq y$ which are not adjacent. We will also write $w_{xy}$ for $w(x,y)$. Similarly, we can extend the map $w'$ of $G'$, and we have
    \begin{multline*}
        (\Delta f)(x) = \frac{1}{m(x)} \sum_{y \in V} w_{xy} (f(y)-f(x)) \\ =
        \frac{1}{m(x)} \sum_{y' \in V'} \sum_{y \in \Phi^{-1}(y')} w_{xy} (f(y)-f(x)) = \frac{1}{m'(x')} \sum_{y' \in V'} \left( \sum_{y \in \Phi^{-1}(y')} w_{xy} \right) (f'(y') - f'(x')) \\ = \frac{1}{m'(x')} \sum_{y' \in V'} w_{x'y'} (f'(y')-f'(x')) = (\Delta' f')(x').
    \end{multline*}
\end{proof}

\begin{rmk} \label{rmk:edgepreim}
    Note that the edge weight function $w': E' \to (0,\infty)$ is a positive function. Therefore, the right hand sum of the relation \eqref{eq:w'wrel} must be non-empty. This implies under the assumptions of the lemma that for every edge $\{x',y'\} \in E'$ and every $x \in \Phi^{-1}(x')$, there must be at least one edge $\{x,y\} \in E$ with $\Phi(y) = y'$.   
\end{rmk}

Now we can state the following general result about curvature relations:

\begin{thm} \label{thm:curvrel}
Let $\Phi: V \to V'$ be a surjective $1$-Lipschitz map between $G=(V,E)$ and $G'=(V',E')$. Let $(m,w)$ and $(m'w')$ be weigthing schemes on $G$ and $G'$, respectively, satisfying the relations \eqref{eq:mm'rel} and \eqref{eq:w'wrel}. Then we have for every edge $\{x',y'\} \in E'$, 
\begin{equation} \label{eq:ollcrel} 
\kappa_{LLY}^G(x',y') \ge \kappa_{LLY}^{G'}(x,y), 
\end{equation}
where $\{x,y\} \in E$ is any edge with $x' = \Phi(x)$ and $y' = \Phi(y)$ (such an edge $\{x,y\}$ exists due to the previous remark).
Moreover, we have for every $x' \in V'$ and every $x \in \Phi^{-1}(x')$,
\begin{equation} \label{eq:becrel} 
K^{G'}(x') \ge K^G(x). 
\end{equation}
\end{thm}

\begin{proof}
    Let $\{x',y'\} \in E'$ and $\{x,y\} \in E$ be edges with $x' = \Phi(x)$ and $y' = \Phi(y)$. Let us first prove the statement \eqref{eq:ollcrel} about the Ollivier Ricci curvature. Note that we have $d(x,y) = d'(x',y') = 1$. 

    Next we observe the following facts about \textrm{\rm{1}--{\rm Lip}} functions: if $f': V' \to \mathbb{R}$ is  \textrm{\rm{1}--{\rm Lip}} in $G'$, then $f = f' \circ \Phi: V \to \mathbb{R}$ is \textrm{\rm{1}--{\rm Lip}} in $G$, since we have for $u,v \in V$,
    $$ |f(u)-f(v)| = |f'(\Phi(u))-f'(\Phi(v))| \le d'(\Phi(u),\Phi(v)) \le d(u,v). $$
    On the other hand, there may be \textrm{\rm{1}--{\rm Lip}} functions $f: V \to \mathbb{R}$ which are not of the form $f = f' \circ \Phi$.

    Bringing these facts together, we conclude with Lemma \ref{lem:laprel} that
    $$ \kappa_{LLY}^{G'}(x',y') = \inf_{\substack{f' \in \textrm{\rm{1}--{\rm Lip}}\\ f'(y')-f'(x') = 1}} \Delta' f'(x') - \Delta' f'(y') \ge \inf_{\substack{f \in \textrm{\rm{1}--{\rm Lip}}\\ f(y)-f(x) = 1}} \Delta f(x) - \Delta f(y) = \kappa_{LLY}^G(x,y), $$
    since the infimum on the right hand side may be over a larger set of \textrm{\rm{1}--{\rm Lip}} functions.

    For the proof of the Bakry-\'Emery curvature relation \eqref{eq:becrel}, we observe the following relations between the bilinear forms
    ${\bf{\Gamma}}, {\bf{\Gamma}}_2$ and ${\bf{\Gamma}'},{\bf{\Gamma}}_2'$ on $G$ and $G'$: We conclude with Lemma \ref{lem:laprel} that, for $f': V' \to \mathbb{R}$, $f = f \circ \Phi$ and $x \in \Phi^{-1}(x')$,
    \begin{multline*}
        {\bf{\Gamma}}'f'(x') = \frac{1}{2} \left( \Delta'((f')^2)(x') - 2 f'(x')(\Delta'f')(x') \right) \\ = 
        \frac{1}{2} \left( \Delta((f)^2)(x) - 2 f(x)(\Delta f)(x) \right) ={\bf{\Gamma}} f(x),
    \end{multline*}
    and similarly,
    $$ {\bf{\Gamma}}_2'f'(x') = {\bf{\Gamma}}_2f(x). $$
    This implies
    $$ K^{G'}(x') = \inf_{f':\,\,\, {\bf{\Gamma}}'f'(x') = 1} {\bf{\Gamma}}_2'f'(x') \ge \inf_{f:\,\,\, {\bf{\Gamma}}f(x) = 1} {\bf{\Gamma}}_2f(x) = K^G(x), $$
    since the infimum on the right hand side might be over a larger set of functions.
\end{proof}

We finish this section by showing that the statement in Theorem \ref{thm:addrel} is a special case of the more general Theorem \ref{thm:curvrel}.

\begin{proof}[Proof of Theorem \ref{thm:addrel}] Let $G=\Cay(\Gamma,S_\Gamma)$, $G'=\Cay(\Gamma',S_{\Gamma'})$ and $\Phi: \Gamma \to \Gamma'$ be defined by $\Phi([w]_R) = [w]_{R'}$. This construction implies that the map $\Phi$ is surjective and maps adjacent vertices $[w]_R$ and $[ws]_R$ in $G$ either to adjacent vertices $[w]_{R'}$ and $[ws]_{R'}$ or to the same vertex $[w]_{R'} = [ws]_{R'}$. This implies that $\Phi$ is a surjective $1$-Lipschitz map. Moreover, every edge incident to the identity element in $G'$ corresponds to a non-trivial equivalence class $[s]_{R'}$ for some $s \in S$, and the assumption $R' \supset R$ implies that the equivalence class $[s]_R$ is also non-trivial. This shows that the statement in Remark \ref{rmk:edgepreim} is satisfied.
The vertex measures $m = \mathbbm{1}_V$ and $m' = \mathbbm{1}_{V'}$ on $G$ and $G'$ satisfy the relation \eqref{eq:mm'rel}, and condition \eqref{eq:w0w0'} implies the edge weight relation \eqref{eq:w'wrel}. 
Therefore, we can apply Theorem \ref{thm:curvrel} and finish the proof.   
\end{proof}

{\bf{Acknowledgements:}} Shiping Liu is supported by the National Key R and D Program of China 2020YFA0713100 and the National Natural Science Foundation of China No. 12031017. David Cushing is supported by the Leverhulme Trust Research Project Grant number RPG-2021-080. Supanat Kamtue is supported
by Shuimu Scholar Program of Tsinghua University No. 2022660219. We like to thank Viola Siconolfi for stimulating discussions on this topic.

\end{document}